\documentclass[12pt]{amsart}

\usepackage[margin=1in]{geometry}

\usepackage{amsmath,amssymb,amsthm,amsfonts,hyperref,url,enumerate}
\usepackage{xcolor}

\usepackage[utf8]{inputenc}
\usepackage[T1]{fontenc}
\usepackage{lmodern}

\newtheorem{theorem}{Theorem}
\newtheorem{corollary}[theorem]{Corollary}
\newtheorem{lemma}[theorem]{Lemma}

\theoremstyle{remark} \newtheorem*{remark}{Remark}

\numberwithin{equation}{section}
\numberwithin{theorem}{section}

\title{Enumerating Galois extensions of number fields}

\author{Robert J. Lemke Oliver}

\begin{document}
\begin{abstract}
	Let $k$ be a number field.  We provide an asymptotic formula for the number of Galois extensions of $k$ with absolute discriminant bounded by some $X \geq 1$, as $X\to\infty$.  We also provide an asymptotic formula for the closely related count of extensions $K/k$ whose normal closure has discriminant bounded by $X$.  The key behind these results is a new upper bound on the number of Galois extensions of $k$ with a given Galois group $G$ and discriminant bounded by $X$; we show the number of such extensions is $O_{[k:\mathbb{Q}],G} (X^{ \frac{4}{\sqrt{|G|}}})$.  This improves over the previous best bound $O_{k,G,\epsilon}(X^{\frac{3}{8}+\epsilon})$ due to Ellenberg and Venkatesh.  In particular, ours is the first bound for general $G$ with an exponent that decays as $|G| \to \infty$.
\end{abstract}

\maketitle

\section{Introduction}

	For any $X \geq 1$, let 
		$
			\mathcal{F}(X)
				:= \{ K/\mathbb{Q} : |\mathrm{Disc}(K)| \leq X\}
		$
	be the set of number fields of any degree whose absolute discriminant is at most $X$.  It is generally expected that there should be some constant $c>0$ so that $\#\mathcal{F}(X) \sim c X$ as $X \to \infty$, but this appears to be far out of reach; to date, the best known upper bound is $\#\mathcal{F}(X) \leq 2 X^{9 (\log\log X)^3}$ \cite[Corollary 1.9]{LO}, which is not even polynomial in $X$.  The primary challenge faced in obtaining a polynomial bound is the consideration of fields of arbitrarily large degree.  
	
	Here, we overcome this difficulty in a natural setting by establishing an asymptotic for the number of Galois extensions of $\mathbb{Q}$, or, in fact, of any number field, and by counting all number fields when ordered by the discriminant of their normal closure.
	
	More concretely, let $k$ be a number field, and for any $X \geq 1$, let
		$
			\mathcal{F}_k^\mathrm{Gal}(X)
				:= \{ K/k \text{ Galois} : |\mathrm{Disc}(K)| \leq X\}.
		$
	
	\begin{theorem}\label{thm:galois-count-intro}
		For any number field $k$, any $X \geq 1$, and any $\epsilon > 0$, we have
			\begin{equation} \label{eqn:galois-count-general}
				\#\mathcal{F}_k^\mathrm{Gal}(X)
					= \frac{\mathrm{Res}_{s=1} \zeta_k(s)}{2^{r_2(k)} \zeta_k(2) |\mathrm{Disc}(k)|^2} X +  O_{k,\epsilon}(X^{1 - \delta_k + \epsilon})
			\end{equation}
		where $\zeta_k(s)$ is the Dedekind zeta function for $k$ and we have set $\delta_k = \frac{1}{2}$ if $[k:\mathbb{Q}]\leq 3$ and $\delta_k = \frac{2}{[k:\mathbb{Q}]+1}$ in general.  If $k=\mathbb{Q}$, then we moreover have
			\begin{equation} \label{eqn:galois-count-Q}
				\#\mathcal{F}_\mathbb{Q}^\mathrm{Gal}(X)
					= \frac{6}{\pi^2} X + P_2(\log X) \cdot X^{1/2} + O\left(X^{1/2} \exp\left( - c \cdot  (\log X)^{3/5} (\log\log X)^{-1/5}\right)\right)
			\end{equation}
		where $P_2$ is an explicitly computable polynomial of degree $2$ and where $c>0$ is an absolute constant.  
	\end{theorem}
	
	We also prove that the same asymptotic formula holds for the closely related set
		$
			\mathcal{F}_k^\mathrm{nc}(X)
				:= \{ K/k : |\mathrm{Disc}(\widetilde{K})| \leq X\},
		$
	where $\widetilde{K}$ is the normal closure of $K/k$, and where we view extensions as living inside a fixed choice of algebraic closure $\overline{k}$.  
	
	\begin{theorem}\label{thm:normal-closure-count-intro}
		For any number field $k$, any $X \geq 1$, and any $\epsilon > 0$, we have
			\begin{equation} \label{eqn:normal-closure-count-general}
				\#\mathcal{F}_k^\mathrm{nc}(X)
					= \frac{\mathrm{Res}_{s=1} \zeta_k(s)}{2^{r_2(k)} \zeta_k(2) |\mathrm{Disc}(k)|^2} X +  O_{k,\epsilon}(X^{1 - \delta_k + \epsilon})
			\end{equation}
		where $\zeta_k(s)$ is the Dedekind zeta function for $k$ and we have set $\delta_k = \frac{1}{2}$ if $[k:\mathbb{Q}]\leq 3$ and $\delta_k = \frac{2}{[k:\mathbb{Q}]+1}$ in general.  If $k=\mathbb{Q}$, then we moreover have
			\begin{equation} \label{eqn:normal-closure-count-Q}
				\#\mathcal{F}_\mathbb{Q}^\mathrm{nc}(X)
					= \frac{6}{\pi^2} X + P_2(\log X) \cdot X^{1/2} + O\left(X^{1/2} \exp\left( - c \cdot  (\log X)^{3/5} (\log\log X)^{-1/5}\right)\right)
			\end{equation}
		where $P_2$ and $c>0$ are the polynomial of degree $2$ and absolute constant appearing in Theorem~\ref{thm:galois-count-intro}, respectively.
	\end{theorem}
	
	The asymptotic for $\mathcal{F}_k^\mathrm{Gal}(X)$ is the first result counting a subset of $\mathcal{F}(X)$ containing fields of arbitrarily large degree.  Moreover, we might argue that $\mathcal{F}_\mathbb{Q}^\mathrm{Gal}(X)$ is the most natural proper subset of $\mathcal{F}(X)$.  Additionally, the asymptotic for $\mathcal{F}_\mathbb{Q}^{\mathrm{nc}}(X)$ is the first result counting \emph{all} number fields, irrespective of degree, by \emph{some} natural invariant (in this case, the discriminant of the normal closure).  
	
	\begin{remark}
		It is often common in the field counting literature to weigh fields inversely to the size of their automorphism group.  For our purposes, it is a strictly harder problem to consider the unweighted version, with a weighted version following straightforwardly.  Similarly, one may wish to consider fields in $\mathcal{F}_k^\mathrm{nc}(X)$ only up to isomorphism and not inside a fixed algebraic closure $\overline{k}$.  This too is a strictly easier problem, as are those problems arising from placing restrictions on the Galois groups (e.g., that they be abelian, nilpotent, or solvable).  
		We comment further on these variations in \S\ref{sec:variants}.
	\end{remark}
	
	The main terms in Theorem \ref{thm:galois-count-intro} and Theorem~\ref{thm:normal-closure-count-intro} will be familiar to experts, as they arise simply from the count of quadratic extensions $K/k$, which was first established over a general number field in \cite[Theorem 4.2]{DatskovskyWright}; in fact, the error terms in \eqref{eqn:galois-count-general} and \eqref{eqn:normal-closure-count-general} arise from the best known error term in this count \cite[Theorem 2]{McgownTucker}.  Similarly, the term $P_2(\log X) \cdot X^{1/2}$ in \eqref{eqn:galois-count-Q} and \eqref{eqn:normal-closure-count-Q} accounts for the known asymptotic number of $C_3$, $C_4$, and $C_2 \times C_2$ extensions of $\mathbb{Q}$ \cite{Cohn,Baily,Maki,Wright,FreiLoughranNewton}.  
	More generally, given a finite group $G$, we may define
		\[
			\mathcal{F}_k(X;G)
				:= \{ K/k \text{ Galois} : \mathrm{Gal}(K/k) \simeq G, |\mathrm{Disc}(K)| \leq X\},
		\]
	and we observe that
		\[
			\mathcal{F}_k^\mathrm{Gal}(X)
				= \bigcup_G \mathcal{F}_k(X;G).
		\]
	A similar expression holds also for $\#\mathcal{F}_k^\mathrm{nc}(X)$, but where the contribution of each group $G$ is weighted by the number of core-free subgroups $H \leq G$.  (Recall that a subgroup is \emph{core-free} if the intersection of its conjugates is trivial.  Every group has at least one core-free subgroup, namely the trivial group.)  
	
	A conjecture of Malle \cite{Malle} predicts that $\#\mathcal{F}_k(X;G) \ll_{k,G,\epsilon} X^{\frac{p}{(p-1)|G|}+\epsilon}$, where $p$ is the least prime dividing $|G|$.  If this conjecture is true, then one should expect the contribution from groups $G$ with order at least $5$ to land in the error terms of Theorems \ref{thm:galois-count-intro} and \ref{thm:normal-closure-count-intro} -- but only provided this bound holds sufficiently uniformly in $G$, as a priori the Minkowksi bound implies that one must consider groups $G$ of order up to a constant times $\log X$.  
	This form of Malle's conjecture is known in $X$ if $G$ is abelian or nilpotent \cite{Maki,Wright,KlunersMalle}, but remains out of reach for general $G$.
	The previous best general bound in the literature on $\#\mathcal{F}_k(X;G)$ is in work of Ellenberg and Venkatesh \cite[Proposition 1.3]{EllenbergVenkatesh}, and is of the form $\#\mathcal{F}_k(X;G) \ll_{k,G,\epsilon} X^{\frac{3}{8}+\epsilon}$ for every $G$ with order at least $5$.  This bound provides a sufficient bound in terms of $X$, but its dependence on the group $G$ is not explicated.  
	
	Thus, the main technical ingredient leading to Theorems \ref{thm:galois-count-intro} and \ref{thm:normal-closure-count-intro} is the following fully explicit bound on $\#\mathcal{F}_k(X;G)$, that additionally incorporates a substantial asymptotic improvement, and whose dependence on the group $G$ is sufficient for our purposes.
	
	\begin{theorem}\label{thm:explicit-galois-bound-intro}
		There is a positive absolute constant $c_1$ such that for any finite group $G$ and any number field $k$, there holds for every $X \geq 1$
			\[
				\#\mathcal{F}_k(X;G)
					\leq e^{d |G|} \cdot (2 d |G|^2)^{ c_1 d |G|^{1/2}} \cdot
						X^{ \frac{6 }{\sqrt{|G|}}}
			\]
		where $d=[k:\mathbb{Q}]$.  Explicitly, we may take 
		$c_1 = 18.5$. 
	\end{theorem}
	
	There is always a balance to be found between aesthetics and sharpness in stating an explicit result.  For Theorem~\ref{thm:explicit-galois-bound-intro} specifically, we have prioritized separating the dependence on $X$ and the parameters $G$ and $k$, and we have not tried to optimize the admissible value of $c_1$.  Our next result provides a bound with a smaller power of $X$, at the expense of making the dependence of the implied constant on the group $G$ and the degree of $k$ inexplicit.
	
	\begin{theorem} \label{thm:uniform-galois-bound-intro}
		For any finite group $G$, any number field $k$, and any $X \geq 1$, there holds
			\[
				\#\mathcal{F}_k(X;G)
					\ll_{[k:\mathbb{Q}],G} X^{\frac{4}{\sqrt{|G|}}}.
			\]
		In fact, there holds $\#\mathcal{F}_k(X;G) \ll_{[k:\mathbb{Q}],G} X^{\frac{c}{\sqrt{|G|}}}$ with $c = \frac{6935}{18\sqrt{9690}} = 3.913\dots$.
	\end{theorem}
	
	Finally, we provide an even stronger bound in terms of $X$, but where the dependence of the implied constant on the base field $k$ is fully inexplicit and is allowed to depend on the discriminant of $k$, for example.
	
	\begin{theorem} \label{thm:optimal-galois-bound-intro}
		Let $c_0 = \frac{863441}{2880 \sqrt{9690}} \approx 3.045$.  For any finite group $G$, any number field $k$, any $X \geq 1$, and any $\epsilon>0$, there holds
			\[
				\#\mathcal{F}_k(X;G)
					\ll_{k,G,\epsilon} X^{\frac{c_0}{\sqrt{|G|}} + \epsilon}.
			\]
	\end{theorem}
	\begin{remark}
		As suggested by their somewhat ad hoc expressions, the constants $c$ and $c_0$ in Theorems~\ref{thm:uniform-galois-bound-intro} and \ref{thm:optimal-galois-bound-intro} arise from analyzing the groups $G$ for which our methods are weakest; the specific constants $c$ and $c_0$ in the statements arise from the sporadic simple group $\mathrm{J}_3$.  
		By imposing constraints on the composition factors of the groups $G$ considered (e.g., that there is no composition factor isomorphic to $\mathrm{J}_3$), it it possible to reduce the value of the constants $c$ and $c_0$, but only slightly.  In particular, it will be necessary to treat groups $G$ with an abelian socle differently than we do below to provide a constant $c_0$ less than $2$.  
		
		Moreover, even for groups as simple as $\mathbb{F}_p \rtimes \mathbb{F}_p^\times$, at present we do not know how to provide an exponent that is $O(|G|^{-\frac{1}{2}-\delta})$ for some $\delta>0$ independent of $p$.  Indeed, such a result would require substantial progress toward the so-called $\ell$-torsion conjecture on class groups, and at a level that is well beyond the scope of existing methods.  Thus, improvements to the qualitative shape of Theorem~\ref{thm:optimal-galois-bound-intro} will require at least a breakthrough in our understanding of class groups of number fields.  
		In this sense, despite the fact that the stated version of Theorem~\ref{thm:optimal-galois-bound-intro} is technically limited by our understanding of the simple group $\mathrm{J}_3$, the shape of the theorem is more fundamentally limited by our understanding of groups with abelian socles and the closely related problem of bounding torsion subgroups of class groups.  See Theorem~\ref{thm:optimal-almost-almost-simple} below.
	\end{remark}
	
	One natural application where Theorem \ref{thm:optimal-galois-bound-intro} is of use is to bounding the number of ``Galois'' polynomials.
	
	\begin{corollary}
		For any $n \geq 5$ and any $H \geq 1$, the number of monic, irreducible, degree $n$ polynomials $f \in \mathbb{Z}[X]$ with coefficients bounded by $H$ in absolute value such that $|\mathrm{Gal}(f)| = n$, is $O_{n,\epsilon}(H^{2c_0 \sqrt{n}  +1 - \frac{2c_0}{\sqrt{n}} +\epsilon})$, where $c_0$ is as in Theorem~\ref{thm:optimal-galois-bound-intro}.
	\end{corollary}
	\begin{proof}
		This follows immediately from Theorem \ref{thm:optimal-galois-bound-intro} and \cite[Theorem 1.3]{LOThorne-Polynomial}, though see also \cite{Bhargava-vdW} for more on the history of this problem.
	\end{proof}
		
	Theorems \ref{thm:explicit-galois-bound-intro}--\ref{thm:optimal-galois-bound-intro} (and also \cite[Proposition 1.3]{EllenbergVenkatesh} before them, though less substantially) make use of the classification of finite simple groups.  Without using the classification, we show it is at least possible to obtain an exponent of $X$ less than $1$.
	
	\begin{theorem}\label{thm:galois-bound-no-cfsg}
		Let $G$ be a finite group of order at least $3$.  Then for any number field $k$ and any $X \geq 1$, there holds for every $\epsilon > 0$
			\[
				\#\mathcal{F}_k(X;G)
					\ll_{[k:\mathbb{Q}],G,\epsilon}  X^{1 - \frac{1}{4|G|}+\epsilon}.
			\]
		This does not rely on the classification of finite simple groups.
	\end{theorem}
	
	Theorem \ref{thm:galois-bound-no-cfsg} essentially follows from the methods of \cite{LO}, as we explain in \S\ref{sec:no-cfsg}, and treats most groups $G$ simultaneously.  For groups of order divisible by either $2$ or $3$, it is sometimes necessary to consider a particular quotient of $G$.  In carrying this out for groups of odd order divisible by $3$, we invoke the Feit--Thompson theorem that groups of odd order are solvable.  The Feit--Thompson theorem is still a rather heavy hammer, so we leave open the question of whether there is a more elementary proof of a bound with an exponent strictly less than $1$ for every group $G$.
	
	In contrast to the proof of Theorem \ref{thm:galois-bound-no-cfsg}, the proof of Theorem \ref{thm:explicit-galois-bound-intro} proceeds by induction, exploiting the minimal normal subgroups of groups $G$.  It is naturally in our treatment of nonabelian minimal normal subgroups that we appeal to the classification of finite simple groups.  
	Beyond that, though, the proofs of Theorems \ref{thm:galois-count-intro} and \ref{thm:normal-closure-count-intro} also rely on a result of Holt \cite{Holt} that provides an upper bound on the number of finite groups of bounded order.  Holt's proof also relies on the classification, and while a weaker result would still afford a proof of our first theorems, we nevertheless expect any proof of Theorems \ref{thm:galois-count-intro} and \ref{thm:normal-closure-count-intro} to necessarily pass through the classification.

\section*{Acknowledgements}

	The author would like to thank Pham Huu Tiep for many useful conversations on earlier versions of the underlying ideas.  He would also like to thank Arul Shankar, Frank Thorne, and Jesse Thorner for useful comments on a preliminary version of this paper.
	
	The author was supported by a grant from the National Science Foundation (DMS-2200760) and by a Simons Foundation Fellowship in Mathematics.

\section{Preliminary reductions and an inductive strategy}

	As indicated above, we approach the proof of Theorems \ref{thm:explicit-galois-bound-intro}--\ref{thm:optimal-galois-bound-intro} by induction.  In this section, we lay out this strategy more precisely, and we indicate how the general problem can be substantially reduced to considering only certain classes of groups $G$.  In particular, unlike previous inductive approaches to related problems that essentially rely on a single normal subgroup, including \cite{KlunersMalle,Kluners,Wang,Alberts}, our focus in this section is on groups with at least two minimal normal subgroups.  We will ultimately show that given such a group $G$, the study of $G$-extensions can be nearly reduced to the study of various $G_i$-extensions, where $G_i$ is some quotient of $G$ with a unique minimal normal subgroup.  The mechanics realizing this reduction are almost trivial, but at least for the purposes of this paper, its impacts are far from it.
	
	The key to our approach is the following lemma.
	\begin{lemma} \label{lem:disjoint-normals}
		Let $G$ be a finite group, and suppose that $G$ admits two normal subgroups $N_1, N_2 \trianglelefteq G$ such that $N_1 \cap N_2$ is trivial.  Then for any number field $k$ and any $X \geq 1$, we have
			\[
				\#\mathcal{F}_k(X;G)
					\leq \#\mathcal{F}_k(X^{\frac{1}{|N_1|}}; G/N_1) \cdot \#\mathcal{F}_k(X^{\frac{1}{|N_2|}}; G/N_2).
			\]
	\end{lemma}
	\begin{proof}
		Suppose that $K \in \mathcal{F}_k(X;G)$.  Let $K^{N_1}$ and $K^{N_2}$ denote the subfields of $K$ fixed by $N_1$ and $N_2$, respectively.  By the assumption that $N_1 \cap N_2$ is trivial, we find that the compositum of $K^{N_1}$ and $K^{N_2}$ must be $K$, and hence that $K$ is determined by these two subfields.  Moreover, since $[K:K^{N_1}] = |N_1|$, we have that $|\mathrm{Disc}(K)| \geq |\mathrm{Disc}(K^{N_1})|^{|N_1|}$, and hence $|\mathrm{Disc}(K^{N_1})| \leq X^{\frac{1}{|N_1|}}$.  Similarly, we have $|\mathrm{Disc}(K^{N_2})| \leq X^{\frac{1}{|N_2|}}$, and the result follows.
	\end{proof}
	
	In particular, any group $G$ admitting two ``large'' normal subgroups satisfying the hypothesis of Lemma \ref{lem:disjoint-normals} may be readily handled by induction.  
	
	\begin{lemma} \label{lem:induction}
		Let $G$ be a finite group, and suppose that $N_1,N_2 \trianglelefteq G$ are normal subgroups with $N_1 \cap N_2 = 1$ and $\frac{1}{\sqrt{|N_1|}} + \frac{1}{\sqrt{|N_2|}} \leq 1$.  Suppose also that there exists constants $c>0$ and $C_1,C_2 > 0$ such that for any number field $k$ and any $X \geq 1$, we have for each $i=1,2$ that
			\[
				\#\mathcal{F}_k(X;G/N_i)
					\leq C_i X^{\frac{c}{\sqrt{|G/N_i|}}}.
			\]
		Then
			\[
				\#\mathcal{F}_k(X;G)
					\leq C_1C_2 X^{\frac{c}{\sqrt{|G|}}}
			\]
		for every number field $k$ and every $X \geq 1$.
	\end{lemma}
	\begin{proof}
		By Lemma \ref{lem:disjoint-normals} and the assumptions of the lemma, we find
			\[
				\#\mathcal{F}_k(X;G)
					\leq C_1C_2 X^{\frac{c}{|N_1| \sqrt{|G/N_1|}} + \frac{c}{|N_2| \sqrt{|G/N_2|}}}
					= C_1C_2 X^{\frac{c}{\sqrt{|G|}} \left( \frac{1}{\sqrt{|N_1|}} + \frac{1}{\sqrt{|N_2|}} \right)}
					\leq C_1C_2 X^{\frac{c}{\sqrt{|G|}}},
			\]
		as claimed.
	\end{proof}
	
	We are particularly interested in the case when $N_1$ and $N_2$ are distinct minimal normal subgroups, for in this case, the condition that $N_1 \cap N_2 = 1$ is automatic.  The following lemma records a characterization of minimal normal subgroups.
	
	\begin{lemma} \label{lem:minimal-normals}
		Let $G$ be a finite group and $N \trianglelefteq G$ a minimal normal subgroup.  Then either $N \simeq \mathbb{F}_p^r$ for some prime $p$ and integer $r \geq 1$ or $N \simeq T^r$ for some nonabelian simple group $T$ and integer $r \geq 1$.
	\end{lemma}
	\begin{proof}
		This is standard, and follows from the fact that any minimal normal subgroup must be characteristically simple.  See \cite[1.7.3]{KurzweilStellmacher}, for example.
	\end{proof}
	
	Notably, if $G$ has two minimal normal subgroups $N_1$ and $N_2$, at least one of which is nonabelian, then we must have $\frac{1}{\sqrt{|N_1|}} + \frac{1}{\sqrt{|N_2|}} \leq \frac{1}{\sqrt{2}} + \frac{1}{\sqrt{60}} < 1$, since any nonabelian simple group has order at least $60$. From this, we find the easy but profoundly useful reduction:
	
	\begin{lemma} \label{lem:reduction}
		Suppose that the conclusions of Theorems \ref{thm:explicit-galois-bound-intro}--\ref{thm:optimal-galois-bound-intro} hold for every finite group $G$ such that either:
			\begin{itemize}
				\item $G$ has a unique minimal normal subgroup $N$, and $N$ is nonabelian; or
				\item all minimal normal subgroups of $G$ are abelian.
			\end{itemize}
		Then the conclusions of Theorems~\ref{thm:explicit-galois-bound-intro}--\ref{thm:optimal-galois-bound-intro} hold for all finite groups $G$.
	\end{lemma}
	\begin{proof}
		For Theorems~\ref{thm:uniform-galois-bound-intro} and \ref{thm:optimal-galois-bound-intro}, this follows directly from Lemma~\ref{lem:induction}, Lemma~\ref{lem:minimal-normals}, and the above discussion.  It also follows for Theorem~\ref{thm:explicit-galois-bound-intro} analogously, but with an additional straightforward computation that we omit.
	\end{proof}
	
	\begin{remark}
		If one wishes to show that there is some constant $c>0$ such that $\#\mathcal{F}_k(X;G) \ll X^{\frac{c}{|G|}}$ for every finite group $G$ (as would follow from Malle's conjecture), then these reductions show that it suffices to prove that such a constant exists for finite groups $G$ with a unique minimal normal subgroup since we trivially have $\frac{1}{2}+\frac{1}{2} \leq 1$.
	\end{remark}
	
	We now split our discussion according to the two cases of Lemma \ref{lem:reduction}.  We begin by treating the case that all minimal normal subgroups of $G$ are abelian.

\section{Groups with abelian minimal normal subgroups}
	\label{sec:abelian-minimal}

	In this section, we prove Theorems~\ref{thm:explicit-galois-bound-intro}--\ref{thm:optimal-galois-bound-intro} for groups $G$ for which all minimal normal subgroups are abelian, assuming by way of induction that these theorems are true for groups of smaller order.  To this end, we begin by recording some necessary bounds on class groups and abelian extensions of number fields.  We then use these results to provide bounds on ``central extensions'' of number fields, which will prove to be the key in our treatment of this case.  In their seminal work on nilpotent groups, Kl\"uners and Malle \cite{KlunersMalle} also crucially exploited central extensions, though the mechanics of their proof are, of necessity, a little different than ours.
	
\subsection{Bounds on class groups and the number of abelian extensions of a number field}
	\label{subsec:class-group-bounds}
	
	We begin by providing a general bound on the class group of a number field.  In situations where the dependence of such a bound on the degree $[k:\mathbb{Q}]$ is irrelevant (as will prove to be the case in the proof of Theorems~\ref{thm:uniform-galois-bound-intro} and \ref{thm:optimal-galois-bound-intro}), the bound $|\mathrm{Cl}(k)| \ll_{[k:\mathbb{Q}],\epsilon} |\mathrm{Disc}(k)|^{\frac{1}{2}+\epsilon}$ will be sufficient.  However, in situations where the degree dependence is relevant (e.g., as will be the case in the proof of Theorem~\ref{thm:explicit-galois-bound-intro}), we will instead make use of the following fully explicit bound.
	
	\begin{lemma}\label{lem:class-group-bound}
		Let $k$ be a number field, let $d:=[k:\mathbb{Q}]$, and let $\mathrm{Cl}(k)$ denote the ideal class group of $k$.  Then 
			\[
				|\mathrm{Cl}(k)| 
					\leq 2\pi \cdot |\mathrm{Disc}(k)|^{3/4}.
			\]
	\end{lemma}
	\begin{proof}
		From Louboutin \cite[Equation (2)]{Louboutin}, we find 
			\begin{align*}
				|\mathrm{Cl}(k)|
					& \leq \frac{w_k}{2R_k} \left(\frac{2}{\pi}\right)^{r_2} \left(\frac{ e \log |\mathrm{Disc}(k)|}{4(n-1)}\right)^{n-1} |\mathrm{Disc}(k)|^{1/2} \\
					& \leq \frac{w_k}{2R_k} \left(\frac{2}{\pi}\right)^{r_2} |\mathrm{Disc}(k)|^{3/4},
			\end{align*}
		where $R_k$ denotes the regulator of $k$ and $w_k$ is the number of roots of unity in $k$.
		Appealing to lower bounds on the ratio $\frac{R_k}{w_k}$ due to Zimmert \cite[Satz 3]{Zimmert} (taking $\gamma=1.5$), the claim follows unless $k$ is imaginary quadratic.  However, if $k$ is imaginary quadratic, then $\frac{w_k}{2R_k} = 1$ unless $k = \mathbb{Q}(\zeta_3)$ or $k = \mathbb{Q}(i)$, and we observe that the claim follows in all cases.
	\end{proof}
	\begin{remark}
		The constants $2\pi$ and $3/4$ in the upper bound $|\mathrm{Cl}(k)| \leq 2\pi |\mathrm{Disc}(k)|^{3/4}$ are not optimal in general.  For example, as Louboutin notes, stronger bounds on the class number follow from combining work of Lenstra \cite{Lenstra} with that of Zimmert \cite[Satz 2]{Zimmert}; the constant $2\pi$ in particular may be replaced by an expression decaying exponentially with the degree of $k$.  For our purposes, however, the exact constants are not of particular importance, and will have little impact on the final statement of Theorem~\ref{thm:explicit-galois-bound-intro}.  We have therefore used a relatively simpler statement with the intent of aiding the clarity of the argument to follow.
	\end{remark}
	
	We now provide for any abelian group $A$, a general bound on the number of $A$-extensions of a fixed number field.  We begin by bounding the number of such extensions with a fixed relative discriminant.
	
	\begin{lemma} \label{lem:abelian-multiplicity-bound}
		Let $A$ be an abelian group, $k$ a number field, and $\mathfrak{D}$ a squarefree ideal of $k$.  Then the number of $A$-extensions of $k$ whose finite ramified primes are exactly those dividing $\mathfrak{D}$ may be bounded by
			\[
				|A|^{3d} \cdot |A[2]|^{r_1(k)} \cdot |\mathrm{Hom}(\mathrm{Cl}(k),A)| \cdot \prod_{\mathfrak{p} \mid \mathfrak{D}} (|A|-1),
			\]
		where $d=[k:\mathbb{Q}]$, $A[2]$ is the $2$-torsion subgroup of $A$, and $r_1(k)$ is the number of real places of $k$.
		
		Similarly, the number of $A$-extensions of $k$ whose finite ramified primes divide $\mathfrak{D}$ (but may not include all prime divisors of $\mathfrak{D}$) may be bounded by
			\[
				|A|^{2d} \cdot |A[2]|^{r_1(k)} \cdot |\mathrm{Hom}(\mathrm{Cl}(k),A)| \cdot \prod_{\mathfrak{p} \mid \mathfrak{D}} |A|.
			\]
	\end{lemma}
	\begin{proof}
		Let $I_k := \mathbb{A}_k/k^\times$ be the id\`ele class group.  By class field theory, $A$-extensions of $k$ correspond to surjective homomorphisms $I_k \to A$.  Recall the fundamental exact sequence
			\[
				0 \to \mathcal{O}_k^\times \to \prod_v \mathcal{O}_v^\times \to I_k \to \mathrm{Cl}(k) \to 0,
			\]
		where the product runs over places $v$ of $k$, and where, for finite $v$, $\mathcal{O}_v^\times$ is the unit group of the ring of integers $\mathcal{O}_v$ of the completion $k_v$.  Since the functor $\mathrm{Hom}(-,A)$ is left exact, we find the exact sequence
			\[
				0 \to \mathrm{Hom}(\mathrm{Cl}(k),A) \to \mathrm{Hom}(I_k,A) \to \prod_v \mathrm{Hom}(\mathcal{O}_v^\times,A).
			\]
		By continuity, given $\rho \in \mathrm{Hom}(I_k,A)$, the image of $\rho$ in $\mathrm{Hom}(\mathcal{O}_v^\times,A)$ can be nontrivial at finitely many places; in fact, the places for which the map is nontrivial are exactly the places at which the extension corresponding to $\rho$ is ramified.  Hence, the number of $A$-extensions ramified exactly at finite primes dividing $\mathfrak{D}$ may be bounded by
			\[
				|\mathrm{Hom}(\mathrm{Cl}(k),A)| \cdot \prod_{\mathfrak{p} \mid \mathfrak{D}} \left( |\mathrm{Hom}(\mathcal{O}_\mathfrak{p}^\times,A)|-1\right) \cdot \prod_{v \mid \infty} |\mathrm{Hom}(\mathcal{O}_v^\times,A)|,
			\]
		while those that are ramified at most at primes dividing $\mathfrak{D}$ may be bounded by essentially the same expression, but without the term $-1$ in the product over finite primes $\mathfrak{p} \mid \mathfrak{D}$.  The product over infinite places $v$ yields the expression $|A[2]|^{r_1(k)}$ in the claim.  Thus, it suffices to bound $|\mathrm{Hom}(\mathcal{O}_\mathfrak{p}^\times,A)|$ for finite primes $\mathfrak{p}$.
		
		Letting $\pi$ be a uniformizer for $\mathcal{O}_\mathfrak{p}$, we have
			\[
				0 \to 1 + \pi \mathcal{O}_\mathfrak{p} \to \mathcal{O}_\mathfrak{p}^\times \to (\mathcal{O}_\mathfrak{p} / \mathfrak{p})^\times \to 0,
			\]
		and so $|\mathrm{Hom}(\mathcal{O}_\mathfrak{p}^\times,A)|$ may be bounded by the product $|\mathrm{Hom}(1+\pi \mathcal{O}_\mathfrak{p},A)| \cdot |\mathrm{Hom}((\mathcal{O}_\mathfrak{p}/\mathfrak{p})^\times,A)|$. Since $(\mathcal{O}_\mathfrak{p}/\mathfrak{p})^\times$ is cyclic, we have $|\mathrm{Hom}((\mathcal{O}_\mathfrak{p}/\mathfrak{p})^\times,A)| \leq |A|$.  
		Moreover, since $1 + \pi \mathcal{O}_\mathfrak{p}$ is a pro-$p$ group (where $p$ is the rational prime lying below $\mathfrak{p}$), $\mathrm{Hom}(1+\pi\mathcal{O}_\mathfrak{p},A)$ will be trivial unless $p \mid |A|$.
		
		By means of the standard log map, we also have
			\[
				0 \to \mu_{p^\infty}(k_\mathfrak{p}) \to 1 + \pi \mathcal{O}_\mathfrak{p} \to \mathcal{O}_\mathfrak{p} \to 0,
			\]
		where $\mu_{p^\infty}$ is the set of $p$-power roots of unity, so that $|\mathrm{Hom}(1+\pi\mathcal{O}_\mathfrak{p},A)| \leq |A_p| \cdot |\mathrm{Hom}(\mathcal{O}_\mathfrak{p},A_p)$, where $A_p$ is the Sylow $p$-subgroup of $A$.  Finally, since $\mathcal{O}_\mathfrak{p} \simeq \mathbb{Z}_p^{f_\mathfrak{p}}$ as an additive group (where $f_\mathfrak{p} := [ \mathcal{O}_\mathfrak{p}/\mathfrak{p} : \mathbb{F}_p]$), we find $|\mathrm{Hom}(\mathcal{O}_\mathfrak{p},A_p)| \leq |A_p|^{f_\mathfrak{p}}$.  Hence, for primes $\mathfrak{p}$ dividing $|A|$, we obtain
			\[
				|\mathrm{Hom}(\mathcal{O}_\mathfrak{p}^\times,A)|
					\leq |A| \cdot |A_p|^{1+f_\mathfrak{p}}
					\leq (|A|-1) \cdot |A_p|^{3f_\mathfrak{p}}.
			\]
		Thus, all told, we see that the number of $A$-extensions whose finite ramified primes are exactly those dividing $\mathfrak{D}$ is bounded by
			\[
				|A[2]|^{r_1(k)} \cdot |\mathrm{Hom}(\mathrm{Cl}(k),A)| \cdot \prod_{\mathfrak{p} \mid \mathfrak{D}} (|A|-1) \cdot \prod_{\mathfrak{p} \mid |A|} |A_p|^{3f_\mathfrak{p}}.
			\]
		Since for any rational prime $p$, we have $\sum_{\mathfrak{p} \mid p} f_\mathfrak{p} = d$, the first result follows.  The second follows analogously.
	\end{proof}
	
	\begin{remark}
		Note that, particularly in the second case of Lemma \ref{lem:abelian-multiplicity-bound}, we have not enforced the condition that the element $\rho \in \mathrm{Hom}(I_k,A)$ is surjective.  As a result, Lemma \ref{lem:abelian-multiplicity-bound} in fact furnishes us with a bound on the number of $A_0$-extensions of $k$ whose finite ramified primes divide $\mathfrak{D}$ across all subgroups $A_0 \leq A$ simultaneously.  We will exploit this fact later.
	\end{remark}
	
	Using this, we now obtain a bound on the set $\#\mathcal{F}_k(X;A)$ for any abelian group $A$ and any number field $k$.
	
	\begin{lemma} \label{lem:abelian-bound}
		Let $k$ be a number field, and let $A$ be an abelian group of rank $r \geq 1$.  Let $p$ be the smallest prime divisor of $|A|$, and set $a = \frac{p-1}{p} |A|$ and $m = d(|A|-1)$.  Then for any $X \geq 1$, we have
			\[
				\#\mathcal{F}_k(X;A)
					\leq \frac{|A|^{3d} \cdot |A[2]|^{r_1(k)}}{(m-1)!}|\mathrm{Cl}(k)|^r X^{\frac{1}{a}} (\log X^{\frac{1}{a}} + m-1)^{m-1} |\mathrm{Disc}(k)|^{-\frac{p}{p-1}},
			\]
		and in particular also
			\[
				\#\mathcal{F}_k(X;A)
					\leq |A|^{3d} \cdot |A[2]|^{r_1(k)} \cdot e^{d(|A|-1)-1} \cdot |\mathrm{Cl}(k)|^r \cdot X^{\frac{2}{a}} \cdot |\mathrm{Disc}(k)|^{-\frac{p}{p-1}}.
			\]
	\end{lemma}
	\begin{proof}
		We first observe that for any $K \in \mathcal{F}_k(X;A)$, the relative discriminant ideal $\mathfrak{D}_{K/k}$ must be $a$-powerful (i.e., the valuation $v_\mathfrak{p}(\mathfrak{D}_{K/k})$ must be at least $a$ for any prime $\mathfrak{p}$ dividing $\mathfrak{D}_{K/k}$).  As a result, if we let $\mathfrak{D}$ be the product of the (finite) primes ramified in $K$, then we must have $|\mathfrak{D}| \leq (X / |\mathrm{Disc}(k)|^{|A|})^{\frac{1}{a}} = X^{\frac{1}{a}} |\mathrm{Disc}(k)|^{-\frac{p}{p-1}}$, where $|\mathfrak{D}|$ denotes the ideal norm of $\mathfrak{D}$.  Hence, by Lemma \ref{lem:abelian-multiplicity-bound}, we find
			\begin{align*}
				\#\mathcal{F}_k(X;A)
					&\leq |A|^{3d} |A[2]|^{r_1(k)} |\mathrm{Hom}(\mathrm{Cl}(k),A)| \sum_{|\mathfrak{D}| \leq X^{\frac{1}{a}} |\mathrm{Disc}(k)|^{-\frac{p}{p-1}}} \tau_{|A|-1}(\mathfrak{D}) \\
					& \leq |A|^{3d} |A[2]|^{r_1(k)} |\mathrm{Hom}(\mathrm{Cl}(k),A)| \sum_{D \leq X^{\frac{1}{a}} |\mathrm{Disc}(k)|^{-\frac{p}{p-1}}} \tau_{d(|A|-1)}(D),
			\end{align*}
		where the final summation runs over positive integers $D$, and where, for any $m \geq 1$, $\tau_m$ denotes the $m$-fold divisor function.  We note that since $\mathrm{Cl}(k)$ is abelian and $A$ is assumed to have rank $r$, we have $|\mathrm{Hom}(\mathrm{Cl}(k),A)| \leq |\mathrm{Cl}(k)|^r$.  Meanwhile, by \cite{Bordelles}, for any integer $m \geq 1$ and any $Q \geq 1$, we have
			\[
				\sum_{D \leq Q} \tau_m(D)
					\leq \frac{Q}{(m-1)!}(\log Q + m - 1)^{m-1},
			\]
		and this yields the first claim.
		
		For the second, we use the standard inequalities $(m-1)! \geq \left(\frac{m-1}{e}\right)^{m-1}$ and $\log(1+x) \leq x$ to find that
			\[
				\frac{1}{(m-1)!} \left(\log X^{\frac{1}{a}} + m-1\right)^{m-1}
					\leq e^{m-1} \left( \frac{\log X}{a (m-1)} + 1\right)^{m-1}
					\leq e^{m-1} X^{\frac{1}{a}}.
			\]
		This completes the proof.
	\end{proof}
	
	Using the first case of Lemma~\ref{lem:abelian-bound}, we also obtain the following inexplicit bound.
	
	\begin{corollary}\label{cor:inexplicit-abelian-bound}
		Let $k$ be a number field and let $A$ be an abelian group of rank $r \geq 1$.  Let $p$ be the smallest prime divisor of $|A|$ and set $a := \frac{p-1}{p} |A|$.  Then for any $X\geq 1$ and any $\epsilon>0$, we have
			\[
				\#\mathcal{F}_k(X;A)
					\ll_{[k:\mathbb{Q}],|A|,\epsilon} |\mathrm{Cl}(k)|^r X^{\frac{1}{a}+\epsilon} |\mathrm{Disc}(k)|^{-\frac{p}{p-1}}.
			\]
	\end{corollary}
	\begin{proof}
		This follows from the first inequality in Lemma~\ref{lem:abelian-bound}.
	\end{proof}
	
\subsection{Bounds on central extensions}
	\label{subsec:central-bounds}
	
	Recall that a subgroup $A$ of a (finite) group $G$ is called \emph{central} if it is contained in the center $Z(G)$.  Such a subgroup $A$ is necessarily a normal abelian subgroup of $G$, and the aim of this section is to provide useful bounds on $\#\mathcal{F}_k(X;G)$ in terms of $\#\mathcal{F}_k(X;G/A)$ in the case that $G$ has a nontrivial central subgroup $A$.
	The key fact we use in this case is that $G$-extensions of a number field $k$ with the same $G/A$-subextension ``differ'' by $A$-extensions of the base field $k$.  This is made precise by the following lemma.
	
	\begin{lemma}\label{lem:central-extensions}
		Let $G$ be a finite group and let $A \leq Z(G)$ (which implies that $A$ is abelian and normal in $G$).  Let $k$ be a number field and $G_k := \mathrm{Gal}(\overline{k}/k)$ the absolute Galois group of $k$.  Let $\rho\colon G_k \rightarrow G$ be a surjective homomorphism, and let $\bar\rho\colon G_k \to G/A$ be the projection of $\rho$.
		
		Then the set of homomorphisms $\rho^\prime \colon G_k \to G$ whose projection $\bar{\rho^\prime}$ equals $\bar\rho$ is in one-to-one correspondence with the set of homomorphisms $\psi \colon G_k \to A$.
	\end{lemma}
	\begin{proof}
		Let $\rho$ and $\rho^\prime$ be as in the statement of the lemma.  We initially define a map $\psi \colon G_k \to A$ by $\psi(\sigma) = \rho^\prime(\sigma) \rho(\sigma)^{-1}$ whose image lands in $A$ by the assumption that $\bar{\rho^\prime}=\bar{\rho}$.  Since $A$ is central, we then find $\psi(\sigma_1\sigma_2) = \rho^\prime(\sigma_1) \rho^\prime(\sigma_2) \rho(\sigma_2)^{-1} \rho(\sigma_1)^{-1} = \rho^\prime(\sigma_1) \rho(\sigma_1)^{-1} \rho^\prime(\sigma_2)\rho(\sigma_2)^{-1} = \psi(\sigma_1)\psi(\sigma_2)$, so that $\psi$ is in fact a homomorphism.  This defines the correspondence in one direction.  For the other, we simply take $\rho^\prime$ to be the product $\rho \psi$.
	\end{proof}

	Using this, we obtain the following pair of lemmas.
	
	\begin{lemma} \label{lem:central-multiplicity}
		Let $G$ be a finite group and let $A \leq Z(G)$.  Let $k$ be a number field, let $F/k$ be a $G/A$-extension, and let $\mathfrak{D}$ be a squarefree ideal of $k$ divisible by every finite prime ramified in $F$.  Then the number of $G$-extensions $K/k$ for which $K^A = F$ and whose finite ramified primes are exactly those dividing $\mathfrak{D}$ may be bounded by
			\[
				e^{\frac{(\log |G|)^2}{\log 2}} \cdot |A|^{2d} \cdot |A[2]|^{r_1(k)} \cdot |\mathrm{Hom}(\mathrm{Cl}(k),A)| \cdot \prod_{\mathfrak{p} \mid \mathfrak{D}} |A|,
			\]
		where $d=[k:\mathbb{Q}]$.
	\end{lemma}
	\begin{proof}
		We first observe that any two surjective homomorphisms $\bar\rho_0,\bar\rho_1\colon G_k \to G/A$ for which $\overline{k}^{\ker \bar\rho_0} = \overline{k}^{\ker \bar\rho_1} = F$ must differ by post-composition of an automorphism of $G/A$.  In particular, there are at most $|\mathrm{Aut}(G/A)|$ choices for the projection $\bar\rho \colon G_k \to G/A$ of a surjective homomorphism $\rho \colon G_k \to G$ for which $K = \overline{k}^{\ker \rho}$ satisfies $K^A = F$.  We fix a choice of projection $\bar \rho$, noting that since $G/A$ is generated by at most $\Omega(|G/A|)$ elements, $|\mathrm{Aut}(G/A)| \leq |G|^{\Omega(|G/A|)}$, so there are at most this many choices of the projection $\bar\rho$.
		
		We now bound the number of surjective homomorphisms $\rho^\prime\colon G_k \to G$ such that $\bar{\rho^\prime} = \bar{\rho}$ and $\overline{k}^{\ker \rho^\prime}$ is ramified exactly at $\mathfrak{D}$.  We may suppose there is at least one such $\rho^\prime$ (which by an abuse of notation, we denote as $\rho$) such that $\overline{k}^{\ker \rho}$ is ramified exactly at $\mathfrak{D}$.  (If not, the class $\bar\rho$ will contribute nothing to the count.)    By Lemma \ref{lem:central-extensions}, we must have $\rho^\prime = \rho \psi$ for some $\psi \colon G_k \to A$.  Let $A_0 \leq A$ denote the image of $\psi$.  In order for $\overline{k}^{\ker \rho^\prime}$ to be ramified only at $\mathfrak{D}$, it follows that the $A_0$-extension of $k$ corresponding to $\psi$ may be ramified only at primes dividing $\mathfrak{D}$.  The result then follows from Lemma \ref{lem:abelian-multiplicity-bound} and the remark following its proof, together with the fact that there are at most $|A_0|^{\Omega(|A_0|)} \leq |G|^{\Omega(|A|)}$ choices for $\psi$ corresponding to a given extension $\overline{k}^{\ker \psi}$.  Noting that $|G|^{\Omega(|G/A|) + \Omega(|A|)} = |G|^{\Omega(|G|)} \leq e^{\frac{(\log |G|)^2}{\log 2}}$, the result follows.
	\end{proof}
	
	\begin{lemma}\label{lem:central-bound}
		Let $G$ be a finite group and let $A \leq Z(G)$ have rank $r$.  Let $k$ be a number field, and suppose that $F/k$ is a $G/A$-extension.  Define $a = \frac{p-1}{p} |G|$, where $p$ is the least prime dividing $|G|$.  Let $m = d|A|$.  Then for any $X \geq 1$, we have
			\begin{align*}
				&\#\{ K \in \mathcal{F}_k(X;G) : K^A = F\} \\
					&\quad\quad\quad \leq  \frac{e^{\frac{(\log |G|)^2}{\log 2}} \cdot |A|^{2d} \cdot |A[2]|^{r_1(k)}}{(m-1)!}\cdot \frac{|\mathrm{Cl}(k)|^r}{|\mathrm{Disc}(k)|^{-\frac{p}{p-1}}} \cdot X^{\frac{1}{a}} \cdot (\log X^{\frac{1}{a}} + m-1)^{m-1},
			\end{align*}
		and hence also
			\[
				\#\{ K \in \mathcal{F}_k(X;G) : K^A = F\}
					\leq e^{d|A|-1} \cdot e^{\frac{(\log |G|)^2}{\log 2}} \cdot |A|^{2d} \cdot |A[2]|^{r_1(k)} \cdot |\mathrm{Cl}(k)|^r \cdot X^{\frac{2}{a}} |\mathrm{Disc}(k)|^{-\frac{p}{p-1}}
			\]
		and
			\[
				\#\{ K \in \mathcal{F}_k(X;G) : K^A = F\}
					\ll_{d,G,\epsilon} |\mathrm{Cl}(k)|^r X^{\frac{1}{a}+\epsilon} |\mathrm{Disc}(k)|^{-\frac{p}{p-1}}.
			\]
	\end{lemma}
	\begin{proof}
		Since the relative discriminant ideal $\mathfrak{D}_{K/k}$ of any extension $K \in \mathcal{F}_k(X;G)$ must be $a$-powerful, by Lemma \ref{lem:central-multiplicity}, we find
			\begin{align*}
				& \#\{ K \in \mathcal{F}_k(X;G) : K^A = F\} \\
					& \quad\quad \leq e^{\frac{(\log |G|)^2}{\log 2}} \cdot |A|^{2d} \cdot |A[2]|^{r_1(k)} \cdot |\mathrm{Hom}(\mathrm{Cl}(k),A)| \cdot \sum_{\substack{ \mathfrak{D}: \\ |\mathfrak{D}| \leq (X/\mathrm{Disc}(k)^{|G|})^\frac{1}{a}}} \tau_{|A|}(\mathfrak{D}).
			\end{align*}
		Proceeding as in the proof of Lemma \ref{lem:abelian-bound}, the result follows.
	\end{proof}
	
	\begin{remark}
		While we find Lemma~\ref{lem:central-bound} convenient for our purposes, we note that there should be versions of Lemma~\ref{lem:central-bound} that obtain a savings in terms of the discriminant of the field $F = K^A$.  However, such arguments must also account for the wild parts of the discriminants of $K$ and $F$, which potentially have a rather larger dependence on the parameters $d$ and $G$ than we care to allow here.  For example, see \cite[Lemma 3.10]{LO} for bounds on the wild part of the discriminant.
	\end{remark}

\subsection{Proof of Theorems~\ref{thm:explicit-galois-bound-intro}--\ref{thm:optimal-galois-bound-intro} when $G$ has an abelian minimal normal subgroup}
	\label{subsec:abelian-proof}

	We are now ready to prove Theorems~\ref{thm:explicit-galois-bound-intro}--\ref{thm:optimal-galois-bound-intro} in the case that $G$ has an abelian minimal normal subgroup $N$.  We may assume by way of induction that the results are known for all groups of order smaller than $G$, in particular any quotient or subgroup of $G$, and over any number field.  
	
	We begin with the proof of Theorem~\ref{thm:uniform-galois-bound-intro}, as we believe it most cleanly conveys the key ideas.  This proof will also essentially at the same time provide a proof of Theorem~\ref{thm:optimal-galois-bound-intro} for reasons to be explained, and will provide a template on which the proof of Theorem~\ref{thm:explicit-galois-bound-intro} will be based.
	
	\begin{proof}[Proof of Theorem~\ref{thm:uniform-galois-bound-intro} when $G$ has an abelian minimal normal subgroup $N$]
		Let $G$ be a finite group, all of whose minimal normal subgroups are abelian.  We suppose by way of induction that there is some constant $c_0>0$ such that for every group $G_0$ of order less than $|G|$, every number field $k$, and every $X \geq 1$, there holds
			\[
				\#\mathcal{F}_k(X;G_0)
					\ll_{[k:\mathbb{Q}],G_0} X^{\frac{c_0}{\sqrt{|G_0|}}}.
			\]
		We consider the constant $c_0$ in some sense as a parameter, with the goal of observing the conditions on $c_0$ that must hold.  Theorem~\ref{thm:uniform-galois-bound-intro} will follow if we show that $c_0=4$ is admissible.  
		
		Note that if $G$ is abelian of rank at most $2$, then Corollary~\ref{cor:inexplicit-abelian-bound} shows that
			\begin{equation} \label{eqn:uniform-rank-two-abelian}
				\#\mathcal{F}_k(X;G)
					\ll_{[k:\mathbb{Q}],G,\epsilon} X^{\frac{2}{|G|} + \epsilon}
					\ll_{[k:\mathbb{Q}],G,\epsilon} X^{\frac{\sqrt{2}}{\sqrt{|G|}} + \epsilon},
			\end{equation}
		which in particular is consistent with any $c_0 > \sqrt{2}$, and may be regarded as a base case for this part of the induction.
		
		By Lemma~\ref{lem:minimal-normals}, there is some prime $p$ and some integer $r \geq 1$ so that $N \simeq \mathbb{F}_p^r$.  Let $C = C_G(N)$ be the centralizer in $G$ of $N$, and note that $G/C$ is naturally a subgroup of $\mathrm{Aut}(N)$.  We consider three possibilities in turn:
			\begin{enumerate}[\em {Case} 1)]
				\item $r=1$, so that $N \simeq \mathbb{F}_p$;
				\item $r \geq 2$ and $|N| \geq \sqrt{|G|}$; and
				\item $r \geq 2$ and $|N| \leq \sqrt{|G|}$.
			\end{enumerate}
		It evidently suffices to prove the theorem in these three cases.  
		
		\emph{Case 1)} Assume that $N \simeq \mathbb{F}_p$.
		We first consider the possibility that $C=G$, i.e. that  $N$ is central in $G$.  Note that this must be the case if $p=2$.  By Lemma~\ref{lem:central-bound}, we find
			\[
				\#\mathcal{F}_k(X;G)
					\ll_{[k:\mathbb{Q}],G,\epsilon} X^{\frac{ p_0}{(p_0-1) |G|}} \#\mathcal{F}_k(X^{1/p};G/N)
					\ll_{[k:\mathbb{Q}],G,\epsilon} X^{\frac{p_0}{(p_0-1)|G|} + \frac{c_0}{\sqrt{p|G|}} + \epsilon}
			\]
		where $p_0$ is the least prime dividing $|G|$.  We have $\frac{p_0}{p_0-1} \leq 2$ and $p \geq 2$, so it suffices to restrict to $c_0$ satisfying
			\[
				\frac{2}{\sqrt{|G|}} + \frac{c_0}{\sqrt{2}} < c_0, \quad \text{i.e.}\quad c_0 < \frac{4+2\sqrt{2}}{\sqrt{|G|}}.
			\]
		This is compatible with any $c_0 > 3$ once $|G| \geq 6$.  However, if $|G| \leq 5$, then $G$ must be abelian, and we may appeal to \eqref{eqn:uniform-rank-two-abelian} instead.
		
		We next consider the possibility that $C=N$ and that $G/C \simeq \mathrm{Aut}(N) \simeq \mathbb{F}_p^\times$, i.e. that $G \simeq \mathbb{F}_p \rtimes \mathbb{F}_p^\times$.  In this case, we have
			\begin{align*}
				\#\mathcal{F}_k(X;G)
					&\leq \sum_{F \in \mathcal{F}_k(X^{1/p}; C_{p-1})} \#\mathcal{F}_F(X;C_p) \\
					&\ll_{[k:\mathbb{Q}],p,\epsilon} X^{\frac{1}{p-1}+\epsilon} \sum_{F \in \mathcal{F}_k(X^{1/p}; C_{p-1})} |\mathrm{Disc}(F)|^{\frac{1}{2}-\frac{p}{p-1}} \\
					&\ll_{[k:\mathbb{Q}],p,\epsilon} X^{\frac{1}{p-1} + \epsilon}.
			\end{align*}
		Since $|G| = p(p-1)$ in this case, this is consistent with any $c_0 > \sqrt{ \frac{3}{2}} \approx 1.224$, which is sufficient.
		
		We may therefore assume that $p \geq 3$ and that either $|C| \geq 2 |N| = 2p$ or $|G/C| \leq \frac{p - 1}{2}$.  These latter two conditions imply that we must have $|G| \leq \frac{1}{2} |C|^2$, and hence that $|C| \geq \sqrt{2 |G|}$.  We also have the bound $|C| \geq \frac{1}{p-1} |G|$ since $[G:C] \leq p-1$.  Appealing to Lemma~\ref{lem:central-bound} and the induction hypothesis, we find  
			\begin{align*}
				\#\mathcal{F}_k(X;G)
					&\leq \sum_{F \in \mathcal{F}_k(X^{\frac{1}{|N|}};G/N)} \#\{ K \in \mathcal{F}_{F^{C/N}}(X;C) : K^N = F\} \\
					&\ll_{[k:\mathbb{Q}],G,\epsilon} X^{\frac{p_0}{(p_0-1)|C|} + \epsilon} \cdot \#\mathcal{F}_k(X^{\frac{1}{|N|}};G/N) \\
					&\ll_{[k:\mathbb{Q}],G,\epsilon} X^{\frac{2}{|C|} + \frac{c_0}{\sqrt{|G| |N|}} + \epsilon},
			\end{align*}
		where $p_0$ is the least prime dividing $|C|$.  When $p \geq 5$, we use the bounds $|C| \geq \sqrt{2 |G|}$ and $|N| \geq 5$ to conclude this is consistent with the desired bound $O_{[k:\mathbb{Q}],G,\epsilon}(X^{\frac{c_0}{\sqrt{|G|}}+\epsilon})$ for any $c_0>0$ satisfying
		\[
			\sqrt{2} + \frac{c_0}{\sqrt{5}} < c_0, \quad \text{i.e.} \quad c_0 > \frac{5\sqrt{2} + \sqrt{10}}{4} \approx 2.558.
		\]
		
		When $p = 3$, we use the bound $|C| \geq |G| / 2$, and see that this is compatible with the desired bound when
			\[
				\frac{4}{\sqrt{|G|}} + \frac{c_0}{\sqrt{3}} < c_0, \quad \text{i.e.} \quad c_0 > \frac{2(3 + \sqrt{3})}{\sqrt{|G|}}.
			\]
		Thus, any $c_0 > 3$ is admissible when $|G| \geq 12$.  When $|G| \leq 9$ and $G$ is abelian, we may use \eqref{eqn:uniform-rank-two-abelian}.  Finally, the nonabelian group of order $6$ is isomorphic to $\mathbb{F}_3 \rtimes \mathbb{F}_3^\times$, a case that was treated earlier.  This completes the proof of Case 1, the upshot of which is that this case is consistent with any constant $c_0 > 3$ arising from the induction hypothesis.
		
		\emph{Case 2)} Here, we are supposing that $N \simeq \mathbb{F}_p^r$ for some $r \geq 2$ and that $|N| \geq \sqrt{|G|}$.  Let $W \leq N$ be a subspace of codimension $1$, so that $W \simeq \mathbb{F}_p^{r-1}$.  As $N$ is abelian, $W$ is normal in $N$, but since we have assumed that $N$ is minimal, it follows that $W$ is corefree in $G$.  (I.e., that $\cap_{g \in G} W^g = 1$.)  In particular, any $K \in \mathcal{F}_k(X;G)$ is determined by its subfield $K^W$, which is a cyclic extension of $K^N$ with degree $p$.  Moreover, $|\mathrm{Disc}(K^W)| \leq |\mathrm{Disc}(K)|^{\frac{1}{|W|}} \leq X^{\frac{p}{|N|}}$.  Appealing to Corollary~\ref{cor:inexplicit-abelian-bound} to bound the number of possibilities for the extension $K^W/K^N$, we therefore find by our induction hypothesis that
			\begin{align*}
				\#\mathcal{F}_k(X;G)
					&\leq \sum_{F \in \mathcal{F}_k(X^{\frac{1}{|N|}}; G/N)} \#\mathcal{F}_F(X^{\frac{p}{|N|}};C_p) \\
					&\ll_{[k:\mathbb{Q}],G,\epsilon}  X^{\frac{p}{(p-1)|N|} + \epsilon} \sum_{F \in \mathcal{F}_k(X^{\frac{1}{|N|}}; G/N)} |\mathrm{Disc}(F)|^{\frac{1}{2}-\frac{p}{p-1}} \\
					&\ll_{[k:\mathbb{Q}],G,\epsilon} X^{\frac{p}{(p-1)|N|} + \max\left\{0, \frac{c_0}{\sqrt{|G| |N|}} + \frac{1}{2|N|} - \frac{p}{(p-1)|N|}\right\} + \epsilon}.
			\end{align*}
		If the maximum in the exponent above is $0$, then we find using our assumption that $|N| \geq \sqrt{|G|}$ that 
			\[
				\frac{p}{(p-1) |N|} \leq \frac{2}{\sqrt{|G|}},
			\]
		which is consistent with the bound $O_{[k:\mathbb{Q}],G,\epsilon}(X^{\frac{c_0}{\sqrt{|G|}}})$ for any $c_0 > 2$.  If the maximum is not $0$, then the bound above becomes $O_{[k:\mathbb{Q}],G,\epsilon}(X^{\frac{c_0}{\sqrt{|G||N|}} + \frac{1}{2 |N|} + \epsilon})$.  Using our assumptions in this case, together with the lower bound $|N| \geq 4$ that must hold if $r \geq 2$, we find that this is consistent with our goal provided that
			\[
				\frac{c_0}{2} + \frac{1}{2} < c_0, \quad \text{i.e.} \quad c_0 > 1.
			\]
		This completes the proof of the second case, with the upshot being that the induction argument follows provided we assume only that $c_0 > 2$.
		
		\emph{Case 3)} As in Case 2), let $W \leq N$ again be a subspace of codimension $1$.  Let $N_G(W)$ denote its normalizer in $G$, which satisfies $[G : N_G(W)] \leq \frac{p^r-1}{p-1} < \frac{|N|}{p-1}$.  Observe that $N/W$ is a normal subgroup of $N_G(W)/W$ isomorphic to $\mathbb{F}_p$. Let $H \leq N_G(W)$ be the subgroup mapping to its centralizer in the quotient, that is, $H := W.C_{N_G(W)/W}(N/W)$.  We have $[N_G(W) : H] \leq p-1$, and hence $[G:H] < |N|$.  By construction, $H$ contains $N$, so for any $K \in \mathcal{F}_k(X;G)$, the subfield $K^H$ is contained in, and determined by, $K^N$.  Additionally, as in Case 2), the extension $K$ will be determined by its subextension $K^W$, which is a cyclic degree $p$ extension of $K^N$.  But unlike Case 2), we now exploit the fact that $K^W$ is also a Galois extension of $K^H$, with $\mathrm{Gal}(K^W/K^N) \simeq N/W$ a central subgroup.  Note that $|\mathrm{Disc}(K^W)| \leq X^{\frac{p}{|N|}}$.
		
		In particular, appealing to Lemma~\ref{lem:central-bound} and letting $p_0$ be the least prime dividing the order of $H/W$, we find
			\begin{align*}
				\#\mathcal{F}_k(X;G)
					&\leq \sum_{F \in \mathcal{F}_k(X^{\frac{1}{|N|}}; G/N)} \#\{ K \in \mathcal{F}_{F^{H/N}}(X^{\frac{p}{|N|}};H/W) : K^{N/W} = F\} \\
					&\ll_{[k:\mathbb{Q}],G,\epsilon} X^{\frac{c_0}{\sqrt{|G|}\sqrt{|N|}} + \frac{p p_0}{(p_0-1) |N||H/W|} + \epsilon} \\
					&\ll_{[k:\mathbb{Q}],G,\epsilon} X^{\frac{c_0}{\sqrt{|G|}\sqrt{|N|}} + \frac{2}{|H|} + \epsilon}.
			\end{align*}
		Since we have $|H| \geq \frac{|G|}{|N|} \geq \sqrt{|G|}$, a computation as above reveals this is sufficient to obtain any $c_0 > 3$ if $|N| \geq 9$.  If $|N| \leq 8$, we instead directly use the bound $|H| \geq \frac{|G|}{|N|}$.  This will be sufficient for the induction argument provided that
			\[
				\frac{c_0}{\sqrt{|N|}} + \frac{2 |N|}{\sqrt{|G|}} < c_0, \quad \text{i.e.} \quad c_0 > \frac{2 |N|}{\sqrt{|G|}} \left(1 - \frac{1}{\sqrt{|N|}}\right)^{-1}.
			\]
		We then observe that if $|N| = 4$, this is sufficient to obtain any $c_0 > 3$ for $|G| \geq 32$, and if $|N| = 8$, this is sufficient for $|G| \geq 72$.  In fact, if $|N| = 8$, then the condition that $|G| \geq 72$ must be satisfied, since we have assumed that $|N| \leq \sqrt{|G|}$, which implies that $|G| \geq 64$.  However, each group of order $64$ is a $2$-group, hence each of its minimal normal subgroups will be of order $2$, not $8$.  Thus, it remains to consider the case that $|N| = 4$.  Here, we must consider groups of order $20$, $24$, and $28$.  Note that we may assume that $N$ is the unique minimal normal subgroup, since any other minimal normal subgroup will either have order at most $3$ (hence we may treat $G$ instead by Case 1), with $r=1$) or will have order $\geq 4$ (in which case Lemma~\ref{lem:induction} implies the desired result).  It thus follows that the only group to consider is the symmetric group $S_4$.  Here, we observe that
			\begin{align*}
				\#\mathcal{F}_k(X;S_4)
					&\leq \sum_{F \in \mathcal{F}_k(X^{1/4}; S_3)} \#\mathcal{F}_F(X^{1/2}; C_2) \\
					&\ll_{[k:\mathbb{Q}],\epsilon} X^{\frac{1}{2}+\epsilon} \sum_{F \in \mathcal{F}_k(X^{1/4};S_3)} |\mathrm{Disc}(F)|^{-\frac{3}{2}} \\
					&\ll_{[k:\mathbb{Q}],\epsilon} X^{\frac{1}{2}+\epsilon},
			\end{align*}
		where the last inequality follows, for example, on using with $p=3$ the bound $\#\mathcal{F}_k(X;\mathbb{F}_p\rtimes \mathbb{F}_p^\times) \ll_{[k:\mathbb{Q}],p,\epsilon} X^{\frac{1}{p-1}+\epsilon}$ proved earlier.  This is sufficient for any $c_0 > \sqrt{6} \approx 2.449$, and concludes the proof of the theorem.
	\end{proof}
	
	\begin{proof}[Proof of Theorem~\ref{thm:optimal-galois-bound-intro} in the case that $G$ has abelian minimal normal subgroups]
		The proof \\ above shows that any exponent $c_0 > 3$ may be obtained in this step of the induction argument, so this follows mutatis mutandis from the above proof.
	\end{proof}
	
	\begin{proof}[Proof of Theorem~\ref{thm:explicit-galois-bound-intro} in the case that $G$ has abelian minimal normal subgroups]
		Let $G$ be a finite group, and let $N \simeq \mathbb{F}_p^r$ be an abelian minimal normal subgroup of $G$.  We follow the same strategy, and use the same notation, as the proof above, but with the modified induction hypothesis that there are constants $c_0,c_1$ such that for every group $G_0$ of order less than $|G|$, every number field $k$, and every $X \geq 1$, 
			\[
				\#\mathcal{F}_k(X;G_0)
					\leq e^{d |G_0|} (2 d |G_0|^2)^{c_1 d |G_0|^{1/2}} \cdot X^{\frac{c_0}{\sqrt{|G_0|}}},
			\]
		with the aim of showing that the same bound holds for $G$.
		As above, we will track the inequalities that must be satisfied by the parameters $c_1$ and $c_0$, with the goal of showing that the claimed values suffice.  We first note that if $G=N=C_p$, then Lemma~\ref{lem:abelian-bound} and Lemma~\ref{lem:class-group-bound} yield
			\[
				\#\mathcal{F}_k(X;C_p)
					\leq 2\pi \cdot e^{dp-d-1} 2^d p^{3d} X^{\frac{2}{p-1}}
					\leq e^{dp} (2p^2)^{3d/2} X^{\frac{2}{p-1}}.
			\]
		This is consistent with the induction hypothesis for any $c_0 \geq 2 \sqrt{2} \approx 2.828$ and any $c_1 \geq \frac{3}{2\sqrt{2}} \approx 1.06$.  In particular, we may assume that $G \ne C_p$ below.
		
		\emph{Case 1)}  We begin by assuming that $r=1$, so that $N = \mathbb{F}_p$.  If $N$ is central, then we may appeal directly to Lemma~\ref{lem:central-bound} and Lemma~\ref{lem:class-group-bound} with our induction hypothesis to see that
			\begin{align*}
				\#\mathcal{F}_k(X;G)
					&\leq e^{dp-1} e^{\frac{ (\log |G|)^2}{\log 2}} 2^d p^{2d} \cdot 2\pi \cdot X^{\frac{2 p_0}{p_0 |G|}} \cdot \#\mathcal{F}_k(X^{1/p}; G/N) \\
					&\leq e^{dp + d |G| / p } \cdot \frac{2\pi}{e} \cdot 2^{d + c_1 d |G|^{1/2} p^{-1/2}}
					\cdot d^{c_1 d |G|^{1/2}p^{-1/2}} \cdot \\
					&\quad\quad \cdot (|G|^2)^{\frac{\log |G|}{2\log 2} + c_1 d |G|^{1/2}p^{-1/2}} \cdot X^{\frac{2p_0}{(p_0-1)|G|} + \frac{c_0}{\sqrt{p |G|}}} \\
					& \leq e^{d|G|} \cdot (2 d |G|^2)^{c_1 d |G|^{1/2}} \cdot X^{\frac{c_0}{\sqrt{|G|}}}
			\end{align*}
		with
		\[
			c_1 \geq \left(\frac{1}{2}+\frac{\log(2\pi/e)}{2\log 2}\right) \left(1 - \frac{1}{\sqrt{2}}\right)^{-1} \approx 3.770, 
		\]
		and with any $c_0 \geq \frac{4}{3}(\sqrt{6}+\sqrt{3}) \approx 5.575$ provided that $|G| \geq 6$.  However, if $|G|=4$, then Lemma~\ref{lem:abelian-bound} yields a stronger bound than these $c_1$ and $c_0$ provide.
		
		We now consider $G \simeq \mathbb{F}_p \rtimes \mathbb{F}_p^\times$ for $p$ odd.  Using Lemma~\ref{lem:abelian-bound}, we find
			\[
				\#\mathcal{F}_k(X ; \mathbb{F}_p \rtimes \mathbb{F}_p^\times)
					\leq e^{dp(p-1)-1} p^{3d(p-1)} \cdot 2\pi \cdot X^{\frac{2}{p-1}} \sum_{F \in \mathcal{F}_k(X^{1/p}; C_{p-1})} |\mathrm{Disc}(F)|^{\frac{3}{4}-\frac{p}{p-1}}.
			\]
		We treat the sum over $F$ via partial summation, splitting into three cases according to whether $p < 13$, $p>13$, or $p=13$.  When $p<13$, using partial summation and Lemma~\ref{lem:abelian-bound}, we find
			\begin{align*}
				\#\mathcal{F}_k(X ; \mathbb{F}_p \rtimes \mathbb{F}_p^\times)
					&\leq e^{dp(p-1)-1} p^{3d(p-1)} \cdot (2\pi)^2 e^{d(p-1)-1} 2^d (p-1)^{3d} \cdot \frac{16}{13-p} \cdot X^{\frac{2}{p-1} + \frac{13-p}{4p(p-1)}} \\
					&\leq e^{dp(p-1)} \left( p(p-1) \right)^{3d(p-1)} (p-1)^{-3d(p-2)} \cdot \frac{(2\pi)^2}{e^{2d+2}} \cdot 2^{d} e^{dp} \cdot \frac{16}{13-p} \cdot X^{\frac{13+7p}{4p(p-1)}} \\
					&\leq e^{d|G|} (2d |G|^2)^{c_1 d |G|^{1/2}} X^{\frac{c_0}{\sqrt{|G|}}}
			\end{align*}
		for any $c_1 \geq \frac{1}{\sqrt{6}} + \frac{3}{\sqrt{6} \cdot \log 2} \approx 2.175$ and any $c_0 \geq \frac{17}{2\sqrt{6}} \approx 3.470$.  Similarly, if $p > 13$, we find
			\begin{align*}
				\#\mathcal{F}_k(X ; \mathbb{F}_p \rtimes \mathbb{F}_p^\times)
					&\leq e^{dp(p-1)-1} p^{3d(p-1)} \cdot (2\pi)^2 e^{d(p-1)-1} 2^d (p-1)^{3d} \cdot \frac{p+3}{p-13} \cdot X^{\frac{2}{p-1}} \\
					&\leq e^{d|G|} (2d|G|^2)^{c_1 d |G|^{1/2}} X^{\frac{c_0}{\sqrt{|G|}}}
			\end{align*}
		for any $c_1 \geq \frac{1}{4\sqrt{17}} + \frac{\sqrt{17}}{4\log 2} \approx 1.547$ and any $c_0 \geq \frac{\sqrt{17}}{2} \approx 2.061$.  Finally, if $p=13$, we find
			\begin{align*}
				\#\mathcal{F}_k(X ; \mathbb{F}_{13} \rtimes \mathbb{F}_{13}^\times)
					&\leq e^{dp(p-1)-1} p^{3d(p-1)} \cdot (2\pi)^2 e^{d(p-1)-1} 2^d (p-1)^{3d} \cdot X^{\frac{1}{12}} \cdot \left(1 + \frac{\log X}{39}\right)\\
					&\leq e^{d|G|} (2d|G|^2)^{c_1 d |G|^{1/2}} X^{\frac{5}{26}}
			\end{align*}
		for any $c_1 \geq \frac{1}{2\sqrt{39}} + \frac{13}{2\sqrt{39} \cdot \log 2} \approx 1.581$ on using that $1 + \frac{\log X}{39} \leq X^{1/39}$ for every $X \geq 1$; this is consistent with any $c_0 \geq \frac{5\sqrt{39}}{13} \approx 2.401$.
		
		Now, again let $C = C_G(N)$, so that we may assume that either $|C| \geq 2p$ or $[G:C] \leq \frac{p-1}{2}$, either of which implies that $|C| \geq \sqrt{2 |G|}$.  Using Lemma~\ref{lem:central-bound}, we find that
			\begin{align*}
				\#\mathcal{F}_k(X;G)
					&\leq e^{d \frac{|G|}{|C|} p -1} |G|^{\frac{\log|G|}{\log 2}} p^{2d \frac{|G|}{|C|}} \cdot 2\pi \cdot e^{d \frac{|G|}{p}} (2d |G|^2p^{-2})^{c_1 d |G|^{1/2} p^{-1/2}} X^{\frac{2p_0}{(p_0-1)|C|} + \frac{c_0}{\sqrt{p|G|}}} \\
					&\leq e^{d|G|} (2d|G|^2)^{c_1 d |G|^{1/2}} X^{\frac{c_0}{\sqrt{|G|}}}
			\end{align*}
		for any 
			\[
				c_1 \geq \left(\frac{1}{\sqrt{2} \cdot \log 2} + \frac{\log (2\pi/e)}{\sqrt{10}}\right) \cdot \left(1-\frac{1}{\sqrt{3}}\right)^{-1} \approx 3.040
			\]
		and any $c_0 \geq 2\sqrt{3} + 2 \approx 5.464$.  Here, we have exploited that we may assume that $|G| \geq 10$ if $p \geq 5$, and that $|G| \geq 12$ if $p=3$.
		
		This completes the treatment of Case 1), with the upshot being that we must assume that $c_0 \geq \frac{4}{3} (\sqrt{6} + \sqrt{3}) \approx 5.575$ and $c_1 \geq 3.771$.
		
		\emph{Case 2)} We assume that $N \simeq \mathbb{F}_p^r$ for some $r\geq 2$ and that $|N| \geq |G|^{1/2}$.  As in the proof of Theorem~\ref{thm:uniform-galois-bound-intro}, by appealing to Lemma~\ref{lem:abelian-bound}, we find
			\[
				\#\mathcal{F}_k(X;G)
					\leq p^{3d |G| / |N|} 2^{d |G|/|N|} e^{dp |G|/|N| - d |G|/|N| -1 } \cdot 2\pi \cdot X^{\frac{2}{(p-1)|N|}} \sum_{F \in \mathcal{F}_k(X^{1/|N|}; G/N)}\!\! |\mathrm{Disc}(F)|^{\frac{3}{4}-\frac{p}{p-1}}.
			\]
		We treat the inner summation via partial summation and our induction hypothesis.  For convenience, let $C = e^{d|G/N|} (2 d |G/N|^2)^{c_1 d |G/N|^{1/2}}$.  If $c_0 / \sqrt{|G/N|} \leq \frac{p}{p-1}-\frac{1}{2}$, then our induction hypothesis implies that $\#\mathcal{F}_k(T;G/N) \leq C T^{\frac{c_0}{\sqrt{|G/N|}}} \leq C T^{\frac{p}{p-1}-\frac{1}{2}}$ for every $T \geq 1$, from which we find that
			\[
				\sum_{F \in \mathcal{F}_k(X^{1/|N|}; G/N)} |\mathrm{Disc}(F)|^{\frac{3}{4}-\frac{p}{p-1}}
					\leq 4C\left(\frac{p}{p-1}-\frac{1}{2}\right)\cdot X^{\frac{1}{4|N|}}
					\leq 6C X^{\frac{1}{4|N|}}.
			\]
		On the other hand, if $c_0 / \sqrt{|G/N|} > \frac{p}{p-1} - \frac{1}{2}$, we find
			\[
				\sum_{F \in \mathcal{F}_k(X^{1/|N|}; G/N)} |\mathrm{Disc}(F)|^{\frac{3}{4}-\frac{p}{p-1}}
					\leq 4C c_0 X^{\frac{c_0}{\sqrt{|G| |N|}}-\frac{p}{(p-1)|N} + \frac{3}{4|N|}}.
			\]
		Assuming that $c_0 \leq 6$, the constant above is at most $24 C$.  Inserting these estimates into our bound on $\mathcal{F}_k(X;G)$, we deduce that
			\[
				\#\mathcal{F}_k(X;G)
					\leq e^{d |G|} (2d |G|^2)^{c_1 d |G|^{1/2}} X^{\frac{c_0}{\sqrt{|G|}}}
			\]
		for any $c_0 \geq \frac{3}{4} (2+\sqrt{2}) \approx 2.560$, and with any
			\[
				c_1 \geq \left(1 + \frac{\log(48\pi/e)}{2\log 2}\right)(1+\sqrt{2}) \approx 13.304
			\]
		provided that $c_0 \leq 6$.
		
		\emph{Case 3)}  We assume that $N \simeq \mathbb{F}_p^r$ for some $r\geq 2$ and that $|N| \leq |G|^{1/2}$.  Let $W$ and $H$ be as in the proof of Theorem~\ref{thm:uniform-galois-bound-intro}.  As there, we may assume that if $N \simeq \mathbb{F}_2^3$, then $|G| \geq 72$,  and that if $N \simeq \mathbb{F}_2^2$, then either $|G| \geq 32$ or $G\simeq S_4$.  Appealing to Lemma~\ref{lem:central-bound}, we find that if $G \not\simeq S_4$ that
			\begin{align*}
				\#\mathcal{F}_k(X;G)
					&\leq e^{d\frac{|G|}{|H|}p-1} |G|^{\frac{\log |G|}{\log 2}} p^{3d \frac{|G|}{|H|}} \cdot 2\pi \cdot X^{\frac{2p_0}{(p_0-1) |H|}} \cdot \#\mathcal{F}_k(X^{1/|N|}; G/N) \\
					&\leq e^{d |G|} (2d|G|^2)^{c_1 d |G|^{1/2}} X^{\frac{c_0}{\sqrt{|G|}}}
			\end{align*}
		for any $c_0 \geq 6$ and any $c_1 \geq 2 \log\left(\frac{2\pi}{e}\right)/\log 2 \approx 2.417$, where $p_0$ is the least prime dividing the order of $H$.
		
		If $G \simeq S_4$, then we use the explicit bound 
			\[
				\#\mathcal{F}_k(X;S_3) 
					\leq e^{6d-1}3^{6d} (2\pi)^2 e^{2d-1}2^{4d} \cdot \frac{8}{5} \cdot X^{\frac{17}{12}}
					= \frac{8}{5}(2\pi)^2 e^{8d-2}3^{6d} 2^{4d} X^{\frac{17}{12}}
			\]
		for all $X\geq 1$ proved earlier (since $S_3 \simeq \mathbb{F}_3 \rtimes \mathbb{F}_3^\times$), Lemma~\ref{lem:abelian-bound}, and partial summation to deduce that
			\begin{align*}
				\#\mathcal{F}_k(X;S_4)
					&\leq 2^{24d}e^{6d-1}\cdot 2\pi \cdot X \cdot \sum_{F \in \mathcal{F}_k(X^{1/4};S_3)} |\mathrm{Disc}(F)|^{-\frac{5}{4}} \\
					&\leq \frac{68}{5} \frac{(2\pi)^3}{e^3} e^{14d} \cdot 2^{28d} \cdot 3^{6d}\cdot X^{\frac{25}{24}} \\
					&\leq e^{d |S_4|} (2d |S_4|^2)^{c_1 d |S_4|^{1/2}} X^{\frac{c_0}{\sqrt{|S_4|}}}
			\end{align*}
		for any $c_0 \geq \frac{25}{\sqrt{24}} \approx 5.103$ and any $c_1 \geq 0.612$.  This completes Case 3), and the proof of Theorem~\ref{thm:explicit-galois-bound-intro} in this case.
	\end{proof}

\section{Groups with a unique nonabelian minimal normal subgroup}

	In this section, by using tools developed in \cite{LO}, we prove Theorems \ref{thm:explicit-galois-bound-intro}--\ref{thm:optimal-galois-bound-intro} in the case that $G$ has a unique minimal normal subgroup $N$ and $N$ is nonabelian.  We begin by stating an explicit bound, which also suffices to prove Theorem~\ref{thm:uniform-galois-bound-intro}.
	
	\begin{theorem} \label{thm:explicit-almost-almost-simple}
		There are constants $c_1,c>0$ such that the following hold.  Let $G$ be a finite group with a unique minimal normal subgroup $N$, and suppose that $N$ is not abelian.  Then for every number field $k$ and every $X \geq 1$, we have
			\[
				\#\mathcal{F}_k(X;G)
					\leq e^{d |G|} (2d |G|^2)^{c_1 d|G|^{1/2}} X^{\frac{c}{\sqrt{|G|}}}
			\]
		where $d = [k:\mathbb{Q}]$.  In fact, we may take any $c_1 \geq 18.5$ and any $c \geq \frac{6935}{18\sqrt{9690}} = 3.913\dots$.
	\end{theorem}
	
	Unlike the case that all of the minimal normal subgroups of $G$ are abelian, here, we can be rather more precise about the groups under consideration.  We will use this description to provide a direct proof of Theorem~\ref{thm:explicit-almost-almost-simple}, rather than proceeding by induction as in the previous section.
	
	\begin{lemma} \label{lem:almost-simple-characterization}
		Let $G$ be a finite group, and suppose that $N = T^r$ is the unique minimal normal subgroup of $G$ for some nonabelian simple group $T$ and some $r \geq 1$.  Then $G$ is isomorphic to a subgroup of $\mathrm{Aut}(T) \wr S_r$ that acts transitively on $T^r$.
	\end{lemma}
	\begin{proof}
		Let $K$ denote the kernel of the map $G \to \mathrm{Aut}(N)$ given by conjugation.  Since $N$ is nonabelian, $N \not \leq K$, and hence $K = 1$ by the assumption that $N$ is the unique minimal normal subgroup of $G$.  It thus follows that $G$ is isomorphic to a subgroup of $\mathrm{Aut}(N) \simeq \mathrm{Aut}(T) \wr S_r$.  The claim about transitivity follows from the fact that $N$ is minimal.
	\end{proof}

	Now, recall that a finite group $G$ is called \emph{almost simple} if it has a unique minimal normal subgroup $N$ and $N=T$ is a nonabelian simple group, and that the subgroup $T$ is referred to as the socle of $G$.  
	Using ideas from \cite{LO}, we first give a bound on $\#\mathcal{F}_k(X;G)$ for almost simple groups $G$.  This will be the key input in the proof of Theorem \ref{thm:explicit-almost-almost-simple}.
	
	\begin{theorem}\label{thm:explicit-almost-simple}
		Let $G$ be an almost simple group.  Then there are constants $n=n(G)$, $a=a(G)$, $w=w(G)$, and $\gamma=\gamma(G)$ such that for any number field $k$ and any $X \geq 1$, there holds
			\[
				\#\mathcal{F}_k(X;G)
					\leq (2\pi)^{dn/2} (\gamma d+1)!^{n} |G|^{d n} (2dn^3)^{dn w} X^{ \frac{a}{\sqrt{|G|}} } |\mathrm{Disc}(k)|^{-\frac{5}{4}},
			\]
		where $d:=[k:\mathbb{Q}]$.  Admissible values of $n(G)$, $a(G)$, $w(G)$, and $\gamma(G)$ are provided in Lemmas~\ref{lem:alternating-bound}--\ref{lem:thompson-bound} (in particular, see Tables~\ref{tab:classical-bounds}--\ref{tab:sporadic-bounds}), and $a(G)=4$ is admissible for every $G$.
	\end{theorem}
	
	The proof of Theorem~\ref{thm:explicit-almost-simple} makes use of the classification of finite simple groups, and in particular proceeds by an analysis of the different socle types $T$.  For the sake of organization, we carry out this proof in the immediately subsequent subsection, before returning to the proof of Theorem~\ref{thm:explicit-almost-almost-simple}.
	
	\subsection{Proof of Theorem~\ref{thm:explicit-almost-simple}}
	
	As the notation and casework involved in the proof may be distracting at first glance, we begin by providing the essential idea of the proof.  We will prove Theorem~\ref{thm:explicit-almost-simple} by instead bounding certain non-Galois extensions $F/k$, whose normal closure has Galois group $G$.  This is made clear by the following lemma.
	
	\begin{lemma} \label{lem:non-galois-passage}
		Let $G$ be a finite group, and let $\pi$ be a faithful and transitive permutation representation of $G$.  Let $n := \deg \pi$.  Then for any $X \geq 1$,
			\[
				\#\mathcal{F}_k(X;G)
					\leq \#\mathcal{F}_{n,k}(X^{\frac{n}{|G|}};\pi(G)),
			\]
		where $\mathcal{F}_{n,k}(X;\pi(G)) := \{ F/k : [F:k]=n, \mathrm{Gal}(\widetilde{F}/k) \simeq_\mathrm{perm} \pi(G)\}$, $\widetilde{F}$ is the normal closure of $F/k$, we view $\mathrm{Gal}(\widetilde{F}/k)$ as a permutation group via its action on the $n$ embeddings $F \hookrightarrow \widetilde{F}$ fixing $k$, and the isomorphism $\mathrm{Gal}(\widetilde{F}/k) \simeq_\mathrm{perm} \pi(G)$ is an isomorphism of permutation groups.
	\end{lemma}
	\begin{proof}
		Let $K \in \mathcal{F}_k(X;G)$ and let $G_0$ be the stabilizer of a point in $\pi(G)$, and note that $[G : G_0] = n$ by the assumption that $\pi$ is transitive.  It follows that the subfield $K^{G_0}=:F$ of $K$ fixed by $G_0$ has degree $n$ over $k$.  By the assumption that $\pi$ is faithful, we find that $\widetilde{F} = K$, so that $K$ is determined by $F$.  Moreover, we find $\mathrm{Gal}(\widetilde{F}/k) \simeq \pi(G)$ and $|\mathrm{Disc}(F)| \leq |\mathrm{Disc}(K)|^{\frac{n}{|G|}} \leq X^{\frac{n}{|G|}}$.  The result follows.
	\end{proof}
	
	Recent work of the author \cite{LO} gives bounds on the sets $\mathcal{F}_{n,k}(X^{\frac{n}{|G|}};\pi(G))$ in terms of the degrees of invariants of $G$ in certain actions (where $\pi$ and $n$ are as in Lemma~\ref{lem:non-galois-passage}).  We make this discussion more precise shortly, but we note at this stage that the quantity $n(G)$ in Theorem~\ref{thm:explicit-almost-simple} will always be the degree of a faithful and transitive permutation representation $\pi$ for which we obtain sufficiently strong bounds on $\#\mathcal{F}_{n,k}(X^{\frac{n}{|G|}};\pi(G))$.
	
	By virtue of the permutation representation $\pi$ in Lemma~\ref{lem:non-galois-passage}, $G$ acts on $\mathbb{Z}[x_1,\dots,x_n]$ by $x_i^g = x_{i^{\pi(g)}}$ and its natural extension, and we refer to an element of $f \in \mathbb{Z}[x_1,\dots,x_n]$ as a $G$-invariant if it is fixed by this action.  We say a $G$-invariant $f$ is homogeneous if it is homogeneous as a polynomial, and we say it is monic if each of its non-zero coefficients is $1$.  Finally, given a set $\mathcal{I}=\{f_1,\dots,f_n\}$ of algebraically independent, monic, homogeneous $G$-invariants, the methods of \cite{LO} give bounds on $\mathcal{F}_{n,k}(X;\pi(G))$ in terms of the degrees of the invariants $f_1,\dots,f_n$.  Combining these methods with Lemma~\ref{lem:non-galois-passage}, we obtain:
	
	\begin{lemma} \label{lem:invariant-theory-bound}
		Let $G$ be an almost simple group and let $\pi$ be a faithful and transitive permutation representation of $G$.  Set $n = \deg \pi$.  Suppose there is a set $\{f_1,\dots,f_n\}$ of $n$ algebraically independent, monic, homogeneous $G$-invariants in $\mathbb{Z}[x_1,\dots,x_n]$.  Then the conclusion of Theorem~\ref{thm:explicit-almost-simple} holds with $n(G) = \deg \pi$, $w(G) = \frac{1}{n} \sum_{i=1}^n \deg f_i$, $\gamma(G)=1$, and
			\[
				a(G)
					= \frac{1}{\sqrt{|G|}} \sum_{i=1}^n \left(\deg f_i - \frac{1}{2}\right).
			\]
	\end{lemma}
	\begin{proof}
		Suppose first that $\sum_{i=1}^n \deg f_i \leq \frac{n(n+1)}{2}$.  Then this follows from \cite[Theorem 3.8]{LO} and Lemma~\ref{lem:non-galois-passage}.  (See also the proof of \cite[Lemma 5.1]{LO}.)  If $\sum_{i=1}^n \deg f_i > \frac{n(n+1)}{2}$, then \cite[Theorem 2.19]{LO} and Lemma~\ref{lem:non-galois-passage} yield a strictly stronger result.  (See also Lemma~\ref{lem:schmidt-bound} below.)
	\end{proof}
	
	Now, for most almost simple groups, it follows from \cite[Theorem 1.10]{LO} that there is always a set of invariants as in Lemma~\ref{lem:invariant-theory-bound} satisfying $\max\{\deg f_i\} \leq C \frac{\log |G|}{\log n}$ for some absolute constant $C$.  (In particular, this holds unless $G$ is an almost simple group of classical type containing the coset of a graph automorphism.  The quantity $\gamma(G)$ will always be $1$ unless $G$ does contain a coset of a graph automorphism.)  We therefore find the qualitative bound $\#\mathcal{F}_k(X;G) \ll_{[k:\mathbb{Q}],G} X^{O\left(\frac{n \log |G|}{ |G| \log n}\right)}$.  As a consequence of the classification of finite simple groups (see Lemma~\ref{lem:almost-simple-degree} below), for every almost simple group $G$, there is a faithful and transitive permutation representation of degree $n \leq \sqrt{|G|}$, which therefore implies the qualitative bound $\#\mathcal{F}_k(X;G) \ll_{[k:\mathbb{Q}],G} X^{O\left(1/{\sqrt{|G|}}\right)}$, which agrees with the claim in Theorem~\ref{thm:explicit-almost-simple}.  
	
	To make the exponent explicit, however, it is necessary to delve deeper into the classification.  For the most part, the results of \cite{LO} are sufficient to establish the bound $a(G) \leq 4$ without additional work (which suffices for both Theorems~\ref{thm:explicit-galois-bound-intro} and \ref{thm:uniform-galois-bound-intro}), but there are a few groups that require a more careful analysis.  Before turning to this analysis, we find it convenient to record two further results along the lines of Lemma~\ref{lem:invariant-theory-bound}.  The first will be useful for treating groups with particularly small degree permutation representations, the second for treating almost simple groups containing the coset of a graph automorphism.
	
	\begin{lemma} \label{lem:schmidt-bound}
		Let $G$ be an almost simple group, let $\pi$ be a faithful and primitive permutation representation of $G$, let $n = \deg \pi$, and suppose that $\pi(G) \leq S_{n_0} \wr S_d$ for some integers $n_0 \geq d \geq 1$ with $n = n_0 d$.  Then the conclusion of Theorem~\ref{thm:explicit-almost-simple} holds for $G$, with $n(G) = n$, $w(G) = 1/6$, $\gamma(G)=1$, and
			\[
				a(G) = \frac{d n_0(n_0+2)}{4\sqrt{|G|}}.
			\]
	\end{lemma}
	\begin{proof}
		This follows from \cite[Theorem 2.19, (2.2)]{LO} and Lemma~\ref{lem:non-galois-passage}.
	\end{proof}
	
	\begin{lemma} \label{lem:induction-bound}
		Let $G$ be an almost simple group containing the coset of a graph automorphism.  Let $G_0$ be the largest subgroup of $G$ not containing the coset of a graph automorphism, and let $\gamma := [G : G_0]$.  Let $\pi_0$ be a faithful and transitive permutation representation of $G_0$ with $\deg \pi_0 =: n_0$ and let $\{f_1,\dots,f_{n_0}\}$ be an algebraically independent set of $G_0$-invariants.  
		
		Then the conclusion of Theorem~\ref{thm:explicit-almost-simple} holds with $n(G) = \gamma n_0$, $\gamma(G) = \gamma$, 
			\[
				w(G) = \frac{1}{n_0} \sum_{i=1}^{n_0} \deg f_i, \text{ and } a(G) = \frac{\gamma}{\sqrt{|G|}} \sum_{i=1}^{n_0} \left(\deg f_i - \frac{1}{2}\right).
			\]
	\end{lemma}
	\begin{proof}
		Let $H_0 \leq G_0$ be such that the coset action of $G_0$ on $H_0$ is isomorphic to $\pi_0$, that is, $H_0$ is the stabilizer in $\pi_0(G_0)$ of a point.  Let $\pi$ be the permutation representation of $G$ corresponding to its action on the cosets of $H_0$, and note that $\deg \pi = \gamma \deg \pi_0$ and that $\pi$ is transitive by construction.  Moreover, because $\pi_0$ is assumed to be faithful, it follows that $\mathrm{Core}_{G_0}(H_0) := \bigcap_{g \in G_0} H_0^g$ is trivial.  Hence we also find that $\mathrm{Core}_G H_0 = 1$, so that $\pi$ is faithful as well.  Thus, by Lemma~\ref{lem:non-galois-passage}, we find
			\[
				\#\mathcal{F}_k(X;G)
					\leq \#\mathcal{F}_{n,k}(X^{\frac{n}{|G|}};\pi(G)).
			\]
		We next observe that any $\pi(G)$-extension of $k$ may be realized as a $\pi_0(G_0)$-extension of an extension of $k$ with degree $\gamma$.  Appealing to \cite[Theorem 3.8]{LO}, \cite[Theorem 2.19]{LO}, and \cite[Proposition 2.18]{LO}, we find the result.
	\end{proof}

	Finally, we record the classification of finite simple groups in the form that we shall use it.
	
	\begin{lemma} \label{lem:cfsg}
		Let $T$ be a nonabelian finite simple group.  Then $T$ is isomorphic to one of the following:
			\begin{enumerate}[i)]
				\item an alternating group $A_n$, for some $n \geq 5$;
				\item a ``classical group'' of the form $\mathrm{PSL}_m(\mathbb{F}_q)$, $\mathrm{PSp}_{2m}(\mathbb{F}_q)$, $\mathrm{PSU}_m(\mathbb{F}_q)$, $\mathrm{P\Omega}^+_{2m}(\mathbb{F}_q)$, $\mathrm{P\Omega}^-_{2m}(\mathbb{F}_q)$, or $\mathrm{P\Omega}_{2m+1}(\mathbb{F}_q)$ for some integer $m \geq 2$ and some prime power $q$;
				\item an ``exceptional group'' of the form $G_2(q)$, $F_4(q)$, $E_6(q)$, $E_7(q)$, $E_8(q)$, $\mathbin{^2E_6}(q^2)$, $\mathbin{^3D_4}(q^3)$, $\mathbin{^2B_2}(2^{2r+1})$, $\mathbin{^2F_4}(2^{2r+1})$, $\mathbin{^2G_2}(3^{2r+1})$, or $\mathbin{^2F_4}(2)^\prime$, where $q$ is a prime power and $r \geq 1$ is an integer; or
				\item one of the $26$ ``sporadic groups'' customarily denoted $\mathrm{M}_{11}$, $\mathrm{M}_{12}$, $\mathrm{M}_{22}$, $\mathrm{M}_{23}$, $\mathrm{M}_{24}$, $\mathrm{J}_1$, $\mathrm{J}_2$, $\mathrm{J}_3$, $\mathrm{HS}$, $\mathrm{McL}$, $\mathrm{Co}_3$, $\mathrm{Co}_2$, $\mathrm{He}$, $\mathrm{Suz}$, $\mathrm{Fi}_{22}$, $\mathrm{Ru}$, $\mathrm{Fi}_{23}$, $\mathrm{J}_4$, $\mathrm{Ly}$, $\mathrm{Co}_1$, $\mathrm{HN}$, $\mathrm{O'N}$, $\mathrm{Th}$, $\mathrm{Fi}_{24}^\prime$, $\mathbb{B}$, or $\mathbb{M}$.
			\end{enumerate}
	\end{lemma}
	
	We comment further on our conventions regarding these groups as necessary.  Additionally, beyond the statement of the classification provided in Lemma~\ref{lem:cfsg}, we will also make use of the properties of these groups (notably the computation of their outer automorphism groups), and, sometimes directly and sometimes indirectly, their character tables and fusion maps as presented in the ATLAS of Finite Groups \cite{ATLAS} and its implementation in GAP \cite{GAP}.
	
	\subsubsection{Alternating groups}
	For alternating groups, Lemma~\ref{lem:schmidt-bound} is sufficient for our purposes.
	
	\begin{lemma} \label{lem:alternating-bound}
		Let $n \geq 5$, let $A_n$ denote the alternating group of degree $n$, and let $G$ be an almost simple group with socle $A_n$.  If $n \ne 6$, or $G \simeq A_6,S_6$, then Theorem~\ref{thm:explicit-almost-simple} holds for $G$, with $n(G) = n$, $w(G) = \frac{1}{2}$, $\gamma(G)=1$, and $a(G) = \frac{n(n+2)}{4 \sqrt{n!/2}} \leq \frac{35}{8\sqrt{15}} < 1.130$.  If $n=6$ and $G \not\simeq A_6,S_6$, then Theorem~\ref{thm:explicit-almost-simple} holds for $G$ with $n(G)=10$, $w(G) = \frac{1}{6}$, $\gamma(G)=1$, and $a(G) = \frac{5}{2\sqrt{5}} < 1.119$.
	\end{lemma}
	\begin{proof}
		Provided that $n \ne 6$, this follows immediately from Lemma~\ref{lem:schmidt-bound} applied to the standard degree $n$ permutation representation of $G$, as does the case that $G \simeq A_6,S_6$.  If $n = 6$ and $G\not\simeq A_6,S_6$, then $G$ has a primitive permutation representation in degree $10$ by virtue of the isomorphism $A_6 \simeq \mathrm{PSL}_2(\mathbb{F}_9)$.  The lemma then follows from Lemma~\ref{lem:schmidt-bound} applied to that degree $10$ representation, and yields that $a(G) = \frac{5}{2\sqrt{5}}$ is admissible for these groups.  Finally, a straightforward computation shows that the largest of these values of $a(G)$ arises for $n=5$, which yields the claim $a(G) < 1.130$ in the lemma.
	\end{proof}
	
	\subsubsection{Classical groups}
	For classical groups, we rely much more heavily on the results of \cite{LO}.  In particular, \cite[Lemmas 4.8--4.17]{LO} provide bounds of the degrees of invariants as in Lemma~\ref{lem:invariant-theory-bound} depending only on the rank of the group and not on the size of the underlying finite field.  Since the minimal degree of a classical group $G$ is typically much smaller than $\sqrt{|G|}$ (in fact, always by a factor at least $O(q^{1/2})$, and typically much more), this provides strong bounds on $a(G)$ for any classical group.  Thus, for each possible classical socle, we provide generic bounds leading to Theorem~\ref{thm:explicit-almost-simple} (which will depend on whether or not there is a graph automorphism), and we identify for each general socle type the almost simple group $G$ for which our methods are weakest (i.e., yield the largest value of $a(G)$).  This latter identification sometimes requires some ad hoc computations, which we describe as necessary.  We perform these computations in Magma \cite{Magma}.  The code used is available at \url{https://lemkeoliver.github.io/}.  Table~\ref{tab:classical-bounds} summarizes the main results for these groups.
	
	\begin{lemma} \label{lem:classical-bounds}
		Let $T$ be a simple classical group, i.e. a group of the form $\mathrm{PSL}_m(\mathbb{F}_q)$, $\mathrm{PSp}_{2m}(\mathbb{F}_q)$, $\mathrm{PSU}_m(\mathbb{F}_q)$, $\mathrm{P\Omega}^+_{2m}(\mathbb{F}_q)$, $\mathrm{P\Omega}_{2m}^-(\mathbb{F}_q)$, or $\mathrm{P\Omega}_{2m+1}(\mathbb{F}_q)$ for some integer $m \geq 2$ and some prime power $q$.  Let $G$ be an almost simple group with socle $T$.  Then Theorem~\ref{thm:explicit-almost-simple} holds for $G$, with the values $n(G)$, $a(G)$, $w(G)$, and $\gamma(G)$ recorded in Table~\ref{tab:classical-bounds}.  Moreover, we may take $a(G) = 2.248$ for every such $G$.
	\end{lemma}
	
	{\small
	\begin{table}[h]
		\begin{tabular}{l|r|r|c|c|r}
			Socle & Conditions & $n(G)$ & $a(G)$ & $w(G)$ & $\gamma(G)$ \\ \hline
			$\mathrm{PSL}_m(\mathbb{F}_q)$ 
				& $G \leq \mathrm{P\Gamma L}_m(\mathbb{F}_q)$ & $\displaystyle \frac{q^m-1}{q-1}$ & $1.453$ & $5m+5$ & $1$\\ 
				& $m \geq 3$, $G \not\leq \mathrm{P\Gamma L}_m(\mathbb{F}_q)$ & $\displaystyle \frac{2(q^m-1)}{q-1}$ & $1.202$ & $5m+5$ & $2$ \\ 
				\hline
			$\mathrm{PSp}_{2m}(\mathbb{F}_q)$ 
				& no graph aut. & $\displaystyle \frac{q^{2m}-1}{q-1}$ & $1.256$ & $7m+5$ & $1$ \\
				& $m=2$, $q$ even, graph aut. & $\displaystyle \frac{2 (q^4-1)}{q-1}$ & $2.248$ & $19$ & $2$ \\ 
				\hline
			$\mathrm{PSU}_m(\mathbb{F}_q)$ 
				& $m=3$ & $\displaystyle q^3+1$ & $2.076$ & $21$ & $1$ \\
				& $m=4$ & $(q^3+1)(q+1)$ & $0.841$ & $25$ & $1$ \\ 
				& $m\geq 5$ odd & $\displaystyle \frac{(q^m+1)(q^{m-1}-1)}{q^2-1}$ & $1.272$ & $\displaystyle \frac{7m+23}{2}$ & $1$ \\
				& $m \geq 6$ even & $\displaystyle \frac{(q^m-1)(q^{m-1}+1)}{q^2-1}$ & $0.214$ & $\displaystyle\frac{7m+18}{2}$ & $1$ \\ 
				\hline
			$\mathrm{P\Omega}^+_{2m}(\mathbb{F}_q)$ 
				& no graph aut. & $\displaystyle \frac{(q^m-1)(q^{m-1}+1)}{q-1}$ & $0.374$ & $7m+9$ & 1 \\
				& $m=4$, graph aut. & $\displaystyle \frac{3(q^4-1)(q^3+1)}{q-1}$ & $0.647$ & $37$ & $3$ \\ 
				\hline
			$\mathrm{P\Omega}^-_{2m}(\mathbb{F}_q)$ & $m \geq 4$ & $\displaystyle \frac{(q^m-1)(q^{m-1}+1)}{q-1}$ & $0.409$ & $7m+15$ & $1$ \\ 
				\hline
			$\mathrm{P\Omega}_{2m+1}(\mathbb{F}_q)$ 
				& $m \geq 3$, $q$ odd & $\displaystyle \frac{q^{2m}-1}{q-1}$ & $0.197$ & $7m+16$ & $1$
			
		\end{tabular}
		\caption{Bounds on classical almost simple groups} \label{tab:classical-bounds}
	\end{table}}
	
	\begin{proof}
		We proceed through the different socle types in turn.
		
		({Linear groups}) Suppose $T = \mathrm{PSL}_m(\mathbb{F}_q)$ for some $m \geq 2$ and some prime power $q$, with $q =8$ or $q\ge 11$ if $m=2$, and with $q \geq 3$ if $m=3$.  (We may make these assumptions since $\mathrm{PSL}_2(\mathbb{F}_4) \simeq \mathrm{PSL}_2(\mathbb{F}_5) \simeq A_5$, $\mathrm{PSL}_2(\mathbb{F}_9) \simeq A_6$, and $\mathrm{PSL}_3(\mathbb{F}_2) \simeq \mathrm{PSL}_2(\mathbb{F}_7)$.  Let $G_0 := G \cap \mathrm{P\Gamma L}_m(\mathbb{F}_q)$, and note that $\gamma(G) = [G:G_0]$.  It follows from \cite[Lemma 4.8]{LO} that in the degree $n_0:=\frac{q^m-1}{q-1}$ permutation representation of $G_0$ on $\mathbb{P}^{m-1}(\mathbb{F}_q)$, there is a set of invariants $\{f_1,\dots,f_{n_0}\}$ with $\deg f_i \leq 5m+5$ for each $i$.  Thus, it follows from Lemma~\ref{lem:induction-bound} that Theorem~\ref{thm:explicit-almost-simple} holds for $G$ with the claimed values of $\gamma(G)$, $w(G)$, and $n(G)$, and with
			\begin{equation} \label{eqn:easy-linear-bound}
				a(G) = \frac{(10m+9)(q^m-1)\gamma(G)}{2(q-1) \sqrt{|G|}}.
			\end{equation}
		We now wish to show that we may also take $a(G) = 1.453$ for every $G \leq \mathrm{P\Gamma L}_m(\mathbb{F}_q)$ and $a(G) = 1.202$ for every $G \not\leq \mathrm{P\Gamma L}_m(\mathbb{F}_q)$.  
		
		Suppose that $G_0=G$.  We first note that \eqref{eqn:easy-linear-bound} is sufficient to obtain $a(G) = 1.453$ if $m \geq 5$, if $m=4$ and $q \geq 3$, if $m=3$ and $q \geq 5$, and if $m=2$ and $q \geq 211$.  We treat the finitely many remaining groups as follows.  First, for the groups with $m=3$ or $m=4$, Lemma~\ref{lem:schmidt-bound} is sufficient and shows that $a(G) = 1.216$ is admissible for these groups.  We therefore restrict our attention to almost simple groups with socle $\mathrm{PSL}_2(\mathbb{F}_q)$ with $q \geq 7$, $q \ne 9$.  
		Using the degree $7$ and $11$ representations of $\mathrm{PSL}_2(\mathbb{F}_7)$ and $\mathrm{PSL}_2(\mathbb{F}_{11})$, respectively, we see that Lemma~\ref{lem:schmidt-bound} is sufficient for $q \leq 11$ and $q=16$.  Next, using the invariants from \cite[Lemma 4.9]{LO} in concert with Lemma~\ref{lem:invariant-theory-bound}, we see that if $q \geq 23$ is prime, or if $q \geq 81$, then we may take 
			\[
				a(G) = \frac{11q+125}{2 \sqrt{|G|}},
			\]
		which is sufficient unless $13 \leq q \leq 43$, $q=49$, or $q=64$.  For the fourteen remaining socle types, we compute a minimal set of algebraically independent invariants in Magma \cite{Magma} as in the proof of \cite[Lemma 4.18]{LO}.  For $G = \mathrm{PSL}_2(\mathbb{F}_{13})$, we find that the minimal set of invariants has degrees $\{1,2,3^2,4^4,5^6\}$.  Combined with Lemma~\ref{lem:invariant-theory-bound}, this shows we may take $a(G) = \frac{24}{\sqrt{273}} = 1.452\dots$ in this case (which gives the claimed value $a(G) = 1.453$).  For the remaining groups $G$, we compute a minimal set of invariants for the group $\mathrm{P\Gamma L}_2(\mathbb{F}_q)$ and verify that they result in a strictly smaller admissible value of $a(G)$.  
		
		Now suppose that $G_0 \ne G$, i.e. that $G \not\leq \mathrm{P\Gamma L}_m(\mathbb{F}_q)$ and $\gamma(G) = 2$.  In this case, \eqref{eqn:easy-linear-bound} is sufficient unless $m=3$ and $q \leq 7$ or $m=4$ and $q=2$. In each of these cases, we compute the bound from Lemma~\ref{lem:schmidt-bound}; the largest value of $a(G)$ arises for a group of the form $G=\mathrm{PSL}_3(\mathbb{F}_4).2$ containing the coset of a graph automorphism.
		
		({Symplectic groups}) Suppose $T=\mathrm{PSp}_{2m}(\mathbb{F}_q)$ for some $m \geq 2$ and some prime power $q$, with $q \geq 3$ if $m=2$.  Let $G_0 = G \cap \mathrm{P\Gamma L}_{2m}(\mathbb{F}_q)$, and note that $\gamma(G) = [G : G_0]$.  It follows from \cite[Lemma 4.10]{LO} that there is a set of $n_0 := \frac{q^{2m}-1}{q-1}$ algebraically independent $G_0$-invariants in its action on $\mathbb{P}^{2m-1}(\mathbb{F}_q)$ with maximum degree at most $7m+5$.  Using Lemma~\ref{lem:induction-bound}, this yields the stated values of $\gamma(G)$, $n(G)$, and $w(G)$, and with 
			\begin{equation} \label{eqn:easy-symplectic-bound}
				a(G) = \frac{(14m+9)(q^{2m}-1) \gamma(G)}{2(q-1) \sqrt{|G|}}.
			\end{equation}
		
		First suppose that $\gamma(G) = 1$.  Then \eqref{eqn:easy-symplectic-bound} implies that $a(G) = 1.256$ is admissible unless $m=2$ and $3 \leq q \leq 5$ or $m=3$ and $q=2$.  If $T = \mathrm{PSp}_6(\mathbb{F}_2)$, then we use Lemma~\ref{lem:schmidt-bound} to see that the claim holds with any $a(G) \leq \frac{455}{64 \sqrt{70}} \approx 0.849$.  Computing a minimal set of invariants for the normalizers of $\mathrm{PSp}_4(\mathbb{F}_3)$ and $\mathrm{PSp}_4(\mathbb{F}_4)$ inside $\mathrm{P\Gamma L}_4(\mathbb{F}_3)$ and $\mathrm{P\Gamma L}_4(\mathbb{F}_4)$, we see that the claim follows for groups of these socle types as well on using Lemma~\ref{lem:invariant-theory-bound}.  Finally, noting that the invariants $\{f_i\}_{i \leq n}$ provided by \cite[Lemma 4.10]{LO} satisfy $\deg f_i = i$ for $i \leq 18$, and $\deg f_i = 19$ for $19 \leq i \leq n$, we obtain using Lemma~\ref{lem:invariant-theory-bound} the bound $a(G) \leq \frac{181}{40\sqrt{13}} = 1.255\dots$ for groups of type $\mathrm{PSp}_4(\mathbb{F}_5)$.
		
		Now suppose that $\gamma(G)=2$, so that $m=2$ and $q \geq 4$ is even.  Then \eqref{eqn:easy-symplectic-bound} is sufficient to obtain the bound $a(G) < 2.248$ unless $q=4$.
		
		(Unitary groups) Suppose $T = \mathrm{PSU}_m(\mathbb{F}_q)$ for some $m \geq 3$ and some prime power $q$, with $q \geq 3$ if $m=3$.  If $m \geq 5$ is odd, then \cite[Lemma 4.13]{LO} implies that there is a set of algebraically independent $G$-invariants in its degree $n(G)$ action with degrees at most $\frac{7m+23}{2}$.  Lemma~\ref{lem:invariant-theory-bound} yields the stated values of $n(G)$, $w(G)$, and $\gamma(G)$, and
			\[
				a(G) =
					\frac{(7m+22)(q^m+1)(q^{m-1}-1)}{2(q^2-1) \sqrt{|G|}}
					\leq \frac{1045}{64 \sqrt{165}} < 1.272,
			\]
		as claimed.  Similarly, if $m \geq 6$ is even, then \cite[Lemma 4.12]{LO} implies that there is a set of algebraically independent $G$-invariants with degrees at most $\frac{7m+18}{2}$.  This yields the claim with the stated values of $n(G)$, $w(G)$, and $\gamma(G)$, and
			\[
				a(G) =
					\frac{(7m+17)(q^m-1)(q^{m-1}+1)}{2(q^2-1)\sqrt{|G|}}
					\leq \frac{4543}{768\sqrt{770}}
					<0.214,
			\]	
		as claimed.
		
		Now suppose that $m=3$.  Using \cite[Lemma 4.13]{LO} and \cite[Corollary 4.4]{LO}, we find that there is a set of invariants of $G$ in its degree $q^3+1$ action with degree at most $21$.  Lemma~\ref{lem:invariant-theory-bound} thus shows that Theorem~\ref{thm:explicit-almost-simple} holds for $G$, with the claimed values of $n(G)$, $w(G)$, and $\gamma(G)$, and with
			\[
				a(G) 
					= \frac{41 (q^3+1)}{2 \sqrt{|G|}}.
			\]
		This is sufficient to obtain the bound $a(G) < 2.056$ unless $q \leq 11$ or $q=17$.
		For groups of type $\mathrm{PSU}_3(\mathbb{F}_3)$ and $\mathrm{PSU}_3(\mathbb{F}_4)$, we compute a minimal set of invariants, which yield a strictly smaller value of $a(G)$.  For $5 \leq q \leq 17$, we experimentally find that groups of type $\mathrm{PSU}_3(q)$ all have a base of size $3$.  Using \cite[Corollary 4.4]{LO}, this implies that there is a set of algebraically independent invariants of degree at most $10$.  Together with Lemma~\ref{lem:invariant-theory-bound}, this is sufficient provided that $G \ne \mathrm{PSU}_3(\mathbb{F}_5)$ and $G \ne 2.\mathrm{PSU}_3(\mathbb{F}_5)$, with the bound for the group $G = \mathrm{PSU}_3(\mathbb{F}_8)$ being largest, namely $a(\mathrm{PSU}_3(\mathbb{F}_8)) = \frac{1083}{32 \sqrt{266}} = 2.075\dots$.  For the groups $\mathrm{PSU}_3(\mathbb{F}_5)$ and $\mathrm{PSU}_3(\mathbb{F}_5).2$, we exploit the fact that such groups have a smaller degree $50$ permutation representation.  We compute a minimal set of invariants for these groups in their smaller representations, which yields a strictly smaller value of $a(G)$.
		
		Finally, suppose that $m=4$.  Using \cite[Lemma 4.14]{LO}, we see that there is a set of independent set of $G$-invariants in its degree $(q^3+1)(q+1)$ representation with degrees at most $25$.  By Lemma~\ref{lem:invariant-theory-bound}, we see that Theorem~\ref{thm:explicit-almost-simple} holds for $G$, with the stated values of $n(G)$, $w(G)$, and $\gamma(G)$, and with
			\[
				a(G) 
					= \frac{49(q^3+1)(q+1)}{2 \sqrt{|G|}}.
			\]
		This is smaller than the claimed value $0.841$ if $q \geq 4$.
		For groups of type $\mathrm{PSU}_4(\mathbb{F}_2)$, we find that a minimal set of invariants has degrees $\{1,2^2,3^4,4^9,5^{11}\}$, which leads to a smaller value of $a(G)$.  For the specific group $T = \mathrm{PSU}_4(\mathbb{F}_3)$, we find that $T$ has a base of size $4$ (and hence independent invariants of degree at most $15$), which leads to $a(T) = \frac{1519}{216\sqrt{70}} = 0.840\dots$.  If $G \leq \mathrm{Aut}(\mathrm{PSU}_4(\mathbb{F}_3))$ properly contains $T$, then we find that $G$ has a base of size $5$.  This yields a smaller value of $a(G)$, completing the proof in this case.
		
		(Orthogonal groups) Suppose $T = \mathrm{P\Omega}_{2m}^+(\mathbb{F}_q)$ for some $m \geq 4$ and some prime power $q$.  Let $G_0$ be the largest subgroup of $G$ not containing the coset of a graph automorphism.  Appealing to \cite[Lemma 4.15]{LO}, there is an independent set of $G_0$-invariants in its degree $\frac{(q^m-1)(q^{m-1}+1)}{q-1}$ action with degrees at most $7m+9$.  Using Lemma~\ref{lem:induction-bound}, we see that Theorem~\ref{thm:explicit-almost-simple} holds with the stated values of $n(G)$, $w(G)$, and $\gamma(G)$, and with
			\[
				a(G)
					=\frac{(14m+17)(q^m-1)(q^{m-1}+1) \gamma}{2(q-1) \sqrt{|G|}}.
			\]
		This expression is maximized when $q=2$, and yields the stated values of $a(G)$.
		
		Suppose now that $T = \mathrm{P\Omega}_{2m}^-(\mathbb{F}_q)$ for some $m \geq 4$ and some prime power $q$.  Then \cite[Lemma 4.17]{LO} implies that there is an independent set of $G$-invariants in its degree $\frac{(q^m+1)(q^{m-1}-1)}{q-1}$ action with degrees at most $7m+15$.  Appealing to Lemma~\ref{lem:invariant-theory-bound}, we find the claim, with
			\[
				a(G)
					= \frac{(14m+29)(q^m-1)(q^{m-1}+1)}{2(q-1)\sqrt{|G|}}
						\leq \frac{1275}{128 \sqrt{595}}
						< 0.409.
			\]
		
		Lastly, suppose that $T = \mathrm{P\Omega}_{2m+1}(\mathbb{F}_q)$ for some $m \geq 3$ and some odd prime power $q$.  Then \cite[Lemma 4.16]{LO} implies that there is a independent set of $G$-invariants in its degree $\frac{q^{2m}-1}{q-1}$ action with degrees at most $7m+16$.  The claim follows from Lemma~\ref{lem:invariant-theory-bound}, with
			\[
				a(G)
					= \frac{(14m+31)(q^{2m}-1)}{2(q-1) \sqrt{|G|}}
					\leq  \frac{6643}{648\sqrt{2730}}
					< 0.197.
			\]
		This completes the proof of the lemma.
	\end{proof}
	
	\subsubsection{Exceptional groups}
	
	We now turn to the treatment of exceptional groups.  This is mostly a matter of bookkeeping, using methods similar to those used for classical groups, but with fewer ad hoc computations being necessary.  Our treatment of the Tits group $\mathbin{^2F_4}(2)^\prime$ may come across as somewhat ad hoc, but previews how we treat certain sporadic groups whose minimal degree permutation representation has degree comparatively close to $\sqrt{|G|}$.  In particular, we view our bounds for it as part of a systematic strategy.  See the discussion preceding Lemma~\ref{lem:thompson-stabilizer} below.
	
	The properties of exceptional groups (their orders and indices of parabolic subgroups in particular) are standard, but we mention the works of Vasil'ev
	 \cite{Vasilyev-G2F4,Vasilyev-E678,Vasilyev-Twisted} 
	 as a convenient reference.  Vasil'ev finds the minimal degree permutation representations of exceptional groups, and these are typically the representations we use in Lemma~\ref{lem:invariant-theory-bound}.  In carrying this out, we will make extensive use of the work of Burness, Liebeck, and Shalev \cite{BLS} that gives upper bounds on the base sizes of such groups in these actions.  (Recall that a base of a permutation group is a collection of points whose stabilizers intersect trivially.)	
	
	\begin{lemma} \label{lem:exceptional}
		Let $T$ be a finite simple group of exceptional Lie type and let $G$ be an almost simple group with socle $T$.  Then Theorem~\ref{thm:explicit-almost-simple} holds for $G$, with the values of $n(G)$, $a(G)$, $w(G)$, and $\gamma(G)$ as recorded in Table~\ref{tab:exceptional-bounds}.  Moreover, we may take $a(G) = 2.374$ for every such $G$.
	\end{lemma}
	
	{\small
	\begin{table}[h]
		\begin{tabular}{l|r|r|c|c|r}
			Socle & Conditions & $n(G)$ & $a(G)$ & $w(G)$ & $\gamma(G)$ \\ \hline
			$G_2(q)$ 
				& $q \geq 3$, no graph aut. & $\displaystyle \frac{q^6-1}{q-1}$ & $1.679$ & $15$ & $1$ \\
				& $q = 3^r$, graph aut. & $\displaystyle \frac{2(q^6-1)}{q-1}$ & $2.374$ & $15$ & $2$ \\
				\hline
			$F_4(q)$
				& no graph aut. & $\displaystyle \frac{(q^{12}-1)(q^4+1)}{q-1}$ & $0.025$ & $21$ & $1$ \\
				& $q=2^r$, graph aut. & $\displaystyle \frac{2(q^{12}-1)(q^4+1)}{q-1}$ & $0.036$ & $21$ & $2$ \\
				\hline
			$E_6(q)$
				& no graph aut. & $\displaystyle \frac{(q^9-1)(q^8+q^4+1)}{q-1}$ & $8.277 \cdot 10^{-6}$ & $28$ & $1$ \\
				& graph aut. & $\displaystyle \frac{2(q^9-1)(q^8+q^4+1)}{q-1}$ & $1.171 \cdot 10^{-5}$ & $28$ & $2$ \\
				\hline
			$E_7(q)$
				& none & $\displaystyle \frac{(q^{14}-1)(q^9+1)(q^5+1)}{q-1}$ & $6.358 \cdot 10^{-11}$ & $21$ & $1$ \\
				\hline
			$E_8(q)$
				& none & $\displaystyle \frac{(q^{30}-1)(q^{12}+1)(q^{10}+1)(q^6+1)}{q-1}$ & $2.313 \cdot 10^{-19}$ & $15$ & $1$ \\
				\hline
			$\mathbin{^2E_6}(q^2)$
				& none & $\displaystyle \frac{(q^{12}-1)(q^6-q^3+1)(q^4+1)}{q-1}$ & $2.080 \cdot 10^{-4}$ & $15$ & $1$ \\
				\hline				
			$\mathbin{^3D_4}(q^3)$
				& none & $\displaystyle (q^8+q^4+1)(q+1)$ & $0.817$ & $15$ & $1$ \\
				\hline
			$\mathbin{^2B_2}(q)$
				& $q=2^{2r+1}$, $r \geq 1$ & $\displaystyle q^2+1$ & $1.708$ & $10$ & $1$ \\
				\hline
			$\mathbin{^2F_4}(q)$
				& $q=2^{2r+1}$, $r \geq 1$ & $\displaystyle (q^6+1)(q^3+1)(q+1)$ & $0.023$ & $10$ & $1$ \\
				\hline
			$\mathbin{^2G_2}(q)$
				& $q=3^{2r+1}$, $r \geq 1$ & $\displaystyle q^3+1$ & $1.864$ & $10$ & $1$ \\
				\hline
			$\mathbin{^2F_4}(2)^\prime$
				& none & $1755$ & $1.872$ & $6$ & $1$ 
		\end{tabular}
		\caption{Bounds on exceptional almost simple groups} \label{tab:exceptional-bounds}
	\end{table} }
	
	\begin{proof}
		We begin with some general considerations to ease our discussion of the case work.  Suppose that a permutation group $G$ of degree $n$ has a base of size $b$.  
		Using \cite[Corollary 4.4]{LO}, it follows that there is a set of $G$-invariants with degrees at most $\frac{(b+1)(b+2)}{2}$.  Using Lemma~\ref{lem:invariant-theory-bound}, it follows that Theorem~\ref{thm:explicit-almost-simple} holds for $G$, with $n(G) = n$, $w(G) = \frac{(b+1)(b+2)}{2}$, $\gamma(G) = 1$, and
			\begin{equation} \label{eqn:base-bound}
				a(G)
					= \frac{(b^2+3b+1) \cdot n}{2 \sqrt{|G|}}.
			\end{equation}
		Similarly, suppose that a finite group $G$ admits an index $\gamma$ subgroup $G_0$, where $G_0$ has a degree $n_0$ permutation representation with a base of size $b$.  Then Theorem~\ref{thm:explicit-almost-simple} holds for $G$, with $n(G) = \gamma n_0$, $w(G) = \frac{(b+1)(b+2)}{2}$, $\gamma(G) = \gamma$, and
			\begin{equation} \label{eqn:base-gamma-bound}
				a(G)
					= \frac{(b^2+3b+1) \cdot n_0\cdot \gamma}{2 \sqrt{|G|}}.
			\end{equation}
		We now use these bounds to handle the different socle types in turn.
		
		If $T=G_2(q)$ for some prime power $q \geq 3$, let $G_0$ be the largest subgroup of $G$ not containing the coset of a graph automorphism.  It follows from \cite[Theorem 3]{BLS} that in its degree $\frac{q^6-1}{q-1}$ action on the cosets of the parabolic subgroup $P_1$, $G_0$ has a base of size at most $4$.  Evaluating the bounds \eqref{eqn:base-bound} and \eqref{eqn:base-gamma-bound}, we obtain smaller values of $a(G)$ than claimed unless $q=3$.  However, if $q=3$, then a computation reveals that $G_2(3)$ has a base of size $3$.  This yields the claimed values of $a(G)$.
		
		If $T=F_4(q)$ for some prime power $q$, let $G_0$ be the largest subgroup of $G$ not containing the coset of a graph automorphism.  By \cite[Theorem 3]{BLS}, there is a base of size at most $5$ for the degree $\frac{(q^12-1)(q^4+1)}{q-1}$ representation of $G_0$ on the cosets of the parabolic $P_1$.  This gives the stated values, with the maximum $a(G)$ arising when $q=2$.
		
		If $T=E_6(q)$ for some prime power $q$, let $G_0$ be the largest subgroup of $G$ not containing the coset of a graph automorphism.  By \cite[Theorem 3]{BLS}, there is a base of size at most $6$ for the degree $\frac{(q^9-1)(q^8+q^4+1)}{q-1}$ action of $G_0$ on the parabolic subgroup $P_1$. This leads to the stated values, with the maximum $a(G)$ arising from $q=2$.
		
		If $T=E_7(q)$ for some prime power $q$, then by \cite[Theorem 3]{BLS}, there is a base of size at most $5$ for the degree $\frac{(q^{14}-1)(q^9+1)(q^5+1)}{q-1}$ action of $G$ on the cosets of the parabolic subgroup $P_1$.  This leads to the stated values, with the maximum again arising from $q=2$.
		
		If $T=E_8(q)$ for some prime power $q$, then by \cite[Theorem 3]{BLS}, there is a base of size $4$ for the degree $\frac{(q^{30}-1)(q^{12}+1)(q^{10}+1)(q^6+1)}{q-1}$ action of $G$ on the cosets of the parabolic subgroup $P_1$.  This leads to the stated values, with the maximum again arising from $q=2$.
		
		If $T=\mathbin{^2E_6}(q^2)$ for some prime power $q$, then by \cite[Theorem 3]{BLS}, there is a base of size $4$ for the degree $ \frac{(q^{12}-1)(q^6-q^3+1)(q^4+1)}{q-1}$ action of $G$ on the cosets of the parabolic subgroup $P_{1,6}$.  This leads to the stated values, with the maximum again arising from $q=2$.
		
		If $T=\mathbin{^3D_4}(q^3)$ for some prime power $q$, then by \cite[Theorem 3]{BLS}, there is a base of size $4$ for the degree $(q^8+q^4+1)(q+1)$ action of $G$ on the cosets of the parabolic subgroup $P_{2}$.  This leads to the stated values, with the maximum again arising from $q=2$.
		
		If $T=\mathbin{^2B_2}(q)$ where $q=2^{2r+1}$ for some integer $r \geq 1$, then by \cite[Theorem 3]{BLS}, there is a base of size $3$ in the degree $q^2+1$ action of $G$ on the cosets of a parabolic subgroup.  This leads to the stated values, provided that $r \geq 2$, with the claimed $a(G)$ arising when $r=2$ and $q=32$.  For the group $\mathbin{^2B_2}(2^3)$ and its automorphism group, we compute explicitly a minimal set of invariants, which yield strictly smaller values of $a(G)$.
		
		If $T=\mathbin{^2F_4}(q)$ where $q=2^{2r+1}$ for some integer $r \geq 1$, then by \cite[Theorem 3]{BLS}, there is a base of size $3$ for the degree $(q^6+1)(q^3+1)(q+1)$ action of $G$ on the cosets of the parabolic subgroup $P_{1}$.  This leads to the stated values, with the maximum arising from $q=8$.
		
		If $T=\mathbin{^2G_2}(q)$ where $q=3^{2r+1}$ for some integer $r \geq 1$, then by \cite[Theorem 3]{BLS}, there is a base of size $3$ for the degree $q^3+1$ action of $G$ on the cosets of a parabolic subgroup.  This leads to the stated values, with the maximum arising from $q=27$.
		
		Finally, if $T=\mathbin{^2F_4}(2)^\prime$ is the Tits group and $G=\mathbin{^2F_4}(2)$ is its automorphism group, then by \cite[Theorem 3]{BLS}, there is a base of size $3$ in the degree $1755$ representation of $G$.  This leads to a larger value of $a(G)$ (namely $3.933$) than we care to allow.  Instead, we verify explicitly that there is a set $\Sigma$ of four points such that $\mathrm{Stab}_G \Sigma = 1$ and $\mathrm{Stab}_G \Sigma^\prime = 1$ for each subset $\Sigma^\prime \subset \Sigma$ of order $3$.  Using \cite[Theorem 4.7]{LO}, we find that there is a set of independent $G$-invariants $\{f_i\}_{i \leq 1755}$ with $\deg f_i = i$ for $i \leq 4$, $\deg f_i = 5$ for $5 \leq i \leq 1743$, and $\deg f_i \leq 9$ for $1744 \leq i \leq 1755$.  These invariants are also $T$-invariants, and the conclusion then follows from Lemma~\ref{lem:invariant-theory-bound}. 
	\end{proof}

	\subsubsection{Sporadic groups}
	We now turn to the treatment of sporadic groups; Table~\ref{tab:sporadic-bounds} provides a summary. For most sporadic groups, we proceed analogously to our treatment of the exceptional groups, relying on work of Burness, O'Brien, and Wilson \cite{BOW} in place of \cite{BLS}.  However, certain groups (namely $J_1$, $J_3$, $J_3.2$, and $\mathrm{Th}$) will be treated similarly to the Tits group.  Our treatment of the Thompson group in particular requires some ideas not present in previous sections.  Thus, we compartmentalize our treatment of the different socle types slightly more than in previous sections.
	
	{\footnotesize \begin{table}[h]
		\begin{tabular}{l|r|c|r|r||r|r|c|r|r} 
			Group & $n(G)$ & $a(G)$ & $w(G)$ & $\gamma(G)$ & 
				Group & $n(G)$ & $a(G)$ & $w(G)$ & $\gamma(G)$ \\ \hline
			$\mathrm{M}_{11}$ & $11$ & $0.402$ & $1/6$ & $1$ &
				$\mathrm{He}.2$ & $2\,058$ & $0.333$ & $15$ & $1$ \\
			$\mathrm{M}_{12}$ & $12$ & $0.137$ & $1/6$ & $1$ &
				$\mathrm{Suz}$ & $1\,782$ & $0.039$ & $15$ & $1$ \\
			$\mathrm{M}_{12}.2$ & $24$ & $0.193$ & $1/6$ & $1$ &
				$\mathrm{Suz}.2$ & $1\,782$ & $0.028$ & $15$ & $1$ \\
			$\mathrm{M}_{22}$ & $22$ & $0.199$ & $1/6$ & $1$ &
				$\mathrm{Fi}_{22}$ & $3\,510$ & $8.96 \cdot 10^{-3}$ & $21$ & $1$ \\
			$\mathrm{M}_{22}.2$ & $22$ & $0.141$ & $1/6$ & $1$ &
				$\mathrm{Fi}_{22}.2$ & $3\,510$ & $8.50 \cdot 10^{-3}$ & $28$ & $1$ \\
			$\mathrm{M}_{23}$ & $23$ & $0.046$ & $1/6$ & $1$ &
				$\mathrm{Ru}$ & $4\,060$ & $0.155$ & $15$ & $1$ \\
			$\mathrm{M}_{24}$ & $24$ & $9.98 \cdot 10^{-3}$ & $1/6$ & $1$ &
				$\mathrm{Fi}_{23}$ & $31\,671$ & $3.22 \cdot 10^{-4}$ & $21$ & $1$\\
			$\mathrm{J}_1$ & $266$ & $2.948$ & $6$ & $1$ &
				$\mathrm{J}_4$ & $173\,067\,389$ & $0.177$ & $10$ & $1$ \\
			$\mathrm{J}_2$ & $100$ & $1.865$ & $15$ & $1$ &
				$\mathrm{Ly}$ & $8\,835\,156$ & $0.369$ & $10$ & $1$ \\
			$\mathrm{J}_2.2$ & $100$ & $1.319$ & $15$ & $1$ &
				$\mathrm{Co}_1$ & $98\,280$ & $9.89 \cdot 10^{-4}$ & $21$ & $1$ \\
			$\mathrm{J}_3$ & $6\,156$ & $3.914$ & $6$ & $1$ &
				$\mathrm{HN}$ & $1\,140\,000$ & $0.656$ & $10$ & $1$ \\
			$\mathrm{J}_3.2$ & $6\,156$ & $2.768$ & $6$ & $1$ &
				$\mathrm{HN}.2$ & $1\,140\,000$ & $0.464$ & $10$ & $1$ \\
			$\mathrm{HS}$ & $100$ & $0.308$ & $21$ & $1$ &
				$\mathrm{O'N}$ & $122\,760$ & $1.718$ & $10$ & $1$ \\
			$\mathrm{HS}.2$ & $100$ & $0.218$ & $21$ & $1$ &
				$\mathrm{O'N}.2$ & $245\,520$ & $2.430$ & $10$ & $2$ \\
			$\mathrm{McL}$ & $275$ & $0.189$ & $21$ & $1$ &
				$\mathrm{Th}$ & $143\,127 \, 000$ & $2.139$ & $6$ & $1$ \\
			$\mathrm{McL}.2$ & $275$ & $0.134$ & $21$ & $1$ &
				$\mathrm{Fi}_{24}^\prime$ & $306\,936$ & $5.62 \cdot 10^{-6}$ & $21$ & $1$ \\
			$\mathrm{Co}_3$ & $276$ & $0.011$ & $28$ & $1$ &
				$\mathrm{Fi}_{24}$ & $306\,936$ & $3.98 \cdot 10^{-6}$ & $21$ & $1$ \\
			$\mathrm{Co}_2$ & $2\,300$ & $9.73 \cdot 10^{-3}$ & $28$ & $1$ &
				$\mathbb{B}$ & $13\,571\,955\,000$ & $3.06 \cdot 10^{-6}$ & $15$ & $1$ \\
			$\mathrm{He}$ & $2\,058$ & $0.471$ & $15$ & $1$ &
				$\mathbb{M}$ & $97\,239\,461\,142\,009\,186\,000$ & $1.03 \cdot 10^{-6}$ & $10$ & $1$ \\
		\end{tabular}
		\caption{Bounds on sporadic almost simple groups}\label{tab:sporadic-bounds}
	\end{table}}
	
	We begin with the Mathieu groups, for which Lemma~\ref{lem:schmidt-bound} will be sufficient.
	
	\begin{lemma} \label{lem:mathieu}
		Let $T$ be one of the Mathieu groups $\mathrm{M}_{11}$, $\mathrm{M}_{12}$, $\mathrm{M}_{22}$, $\mathrm{M}_{23}$, and $\mathrm{M}_{24}$, and let $G$ be an almost simple group with socle $T$.  Then Theorem~\ref{thm:explicit-almost-simple} holds for $G$, with the values $n(G)$, $a(G)$, $w(G)$, and $\gamma(G)$ as recorded in Table~\ref{tab:sporadic-bounds}.  Moreover, we may take $a(G) = 0.402$ for each such $G$.
	\end{lemma}
	\begin{proof}
		If $G$ is not simple, then either $T=\mathrm{M}_{12}$ and $G = \mathrm{M}_{12}.2$ (which has a degree $24$ imprimitive permutation representation) or $T=\mathrm{M}_{22}$ and $G=\mathrm{M}_{22}.2$ (which has a degree $22$ primitive permutation representation).  The lemma then follows from Lemma~\ref{lem:schmidt-bound}.
	\end{proof}
	
	We now treat the socles for which the results from \cite{BOW} on base sizes of sporadic groups will be sufficient.
	
	\begin{lemma} \label{lem:sporadic-base}
		Let $G$ be an almost simple group with sporadic socle not equal to a Mathieu group, $\mathrm{J}_1$, $\mathrm{J}_3$, or the Thompson group $\mathrm{Th}$.  Then Theorem~\ref{thm:explicit-almost-simple} holds for $G$, with the values $n(G)$, $a(G)$, $w(G)$, and $\gamma(G)$ as recorded in Table~\ref{tab:sporadic-bounds}.  Moreover, we may take $a(G) = 2.430$ for each such $G$.
	\end{lemma}
	\begin{proof}
		The work of Burness, O'Brien, and Wilson \cite{BOW} mentioned earlier determines exactly the minimal base of each such group in its minimal degree primitive permutation representation.  For all but $\mathrm{O'N}.2$, we use this minimal degree primitive representation in concert with \eqref{eqn:base-bound} to conclude.  For $\mathrm{O'N}.2$, we use \eqref{eqn:base-gamma-bound}.
	\end{proof}
	
	It therefore remains to treat almost simple groups with socle $\mathrm{J}_1$, $\mathrm{J}_3$, or $\mathrm{Th}$.  For $\mathrm{J}_1$ and $\mathrm{J}_3$, we approach this largely computationally.
	
	\begin{lemma} \label{lem:J1-J3}
		Let $G$ be $\mathrm{J}_1$, $\mathrm{J}_3$, or $\mathrm{J}_3.2$.  Then Theorem~\ref{thm:explicit-almost-simple} holds for $G$, with the values $n(G)$, $a(G)$, $w(G)$, and $\gamma(G)$ as recorded in Table~\ref{tab:sporadic-bounds}.  Moreover, we may take $a(G) = \frac{6395}{18\sqrt{9690}} = 3.913\dots$ for these $G$.
	\end{lemma}
	\begin{proof}
		We explicitly verify in Magma that in the minimal degree primitive permutation representation of these groups, there is a set $\Sigma$ of $4$ points such that $\mathrm{Stab}_G \Sigma = 1$ and so that $\mathrm{Stab}_G \Sigma^\prime = 1$ for each subset $\Sigma^\prime \subset \Sigma$ of order $3$.  It follows from \cite[Theorem 4.7]{LO} that there is then a set of independent invariants $\{f_i\}$ with degrees $\deg f_i = i$ for $i \leq 4$, $\deg f_i = 5$ for $5 \leq i \leq n - 12$, and $\deg f_i \leq 9$ for $n-11 \leq i \leq n$.  The claim now follows from Lemma~\ref{lem:invariant-theory-bound} and an easy computation.
	\end{proof}
	
	We now place this proof in somewhat greater context.  If a permutation group $G$ has a minimal base of size $b$, the best the general methods of \cite{LO} can do is to produce a set of independent $G$-invariants with degrees bounded by $b+2$ apart from a small number of exceptions.  This is achieved by \cite[Theorem 4.7]{LO} precisely when there is a set $\Sigma$ of $b+1$ points such that $\mathrm{Stab}_G \Sigma = 1$ and so that every subset $\Sigma^\prime \subset \Sigma$ of order $b$ forms a base for $G$.  The groups $\mathrm{J}_1$, $\mathrm{J}_3$, and $\mathrm{J}_3.2$ considered in Lemma~\ref{lem:J1-J3} each have a minimal base of size $3$, which explains the role of the set $\Sigma$ with $4$ points.  Similarly, the Tits group $\mathbin{^2F_4}(2)^\prime$ also has a minimal base of size $3$ in its degree $1755$ representation, which explains our treatment of it in the proof of Lemma~\ref{lem:exceptional}.  In fact, the Thompson group $\mathrm{Th}$ also has a minimal base of size $3$ in its minimal degree permutation representation, but unlike these other groups, this permutation representation is not stored in computer algebra systems like GAP or Magma.  We therefore instead provide a theoretical argument that such a set $\Sigma$ must exist.
	
	\begin{lemma} \label{lem:thompson-stabilizer}
		Let $G = \mathrm{Th}$ be the Thompson group of order $90{\,}745{\,}943{\,}887{\,}872{\,}000 = 2^{15} \cdot 3^{10}\cdot 5^3 \cdot 7^2 \cdot 13 \cdot 19 \cdot 31$, in its minimal degree $n = 143\,127\,000$ primitive permutation repreresentation.  Then there is a subset $\Sigma \subseteq \{ 1, \dots, n\}$ of size $4$ such that $\mathrm{Stab}_G \Sigma = 1$ and so that $\mathrm{Stab}_G \Sigma^\prime = 1$ for each subset $\Sigma^\prime \subset \Sigma$ of size $3$.
	\end{lemma}
	\begin{proof}
		We proceed in a manner heavily inspired by the probabilistic methods of \cite{BOW}.  We consider the probability that a randomly chosen set $\Sigma$ does not have this property.  If this probability is strictly less than $1$, then there must be such a $\Sigma$.  Our goal, therefore, is to establish this inequality.
		
		Let $x,y,z,w$ be chosen uniformly at random from $\{1,\dots,n\}$, allowing repetition.  By the union bound, it follows that the probability that $\mathrm{Stab}_G\{ x, y, z, w\} \ne 1$, $\mathrm{Stab}_G\{ x, y, z\} \ne 1$, $\mathrm{Stab}_G\{ x, y, w\} \ne 1$, $\mathrm{Stab}_G\{ x, z, w\} \ne 1$,  or $\mathrm{Stab}_G\{ y, z, w\} \ne 1$ is at most
			\begin{equation} \label{eqn:union-bound}
				\mathbf{Prob}[ \mathrm{Stab}_G\{ x, y, z, w\} \ne 1] + 4\cdot\mathbf{Prob}[ \mathrm{Stab}_G\{ x, y, z\} \ne 1].
			\end{equation}
		Next, if either $\mathrm{Stab}_G\{x,y,z,w\}$ or $\mathrm{Stab}_G\{x,y,z\}$ is nontrivial, then there must be an element of prime order inside the stabilizer.  It follows that the probability the stabilizer is nontrivial is bounded by the expected number of elements of prime order in the stabilizer.  We therefore turn to evaluating this expectation, beginning with $\mathrm{Stab}_G\{x,y,z\}$.
		
		First, recording the above discussion concretely, we find that
			\begin{align*}
				\mathbf{Prob}\left[ \mathrm{Stab}_G\{ x, y, z\} \ne 1\right]
					&= \frac{1}{n^3} \sum_{x,y,z \leq n} \mathbf{1}(\mathrm{Stab}_G\{x,y,z\} \ne 1) \\
					&\leq \frac{1}{n^3}\sum_{x,y,z \leq n} \sum_{ \substack{ g \in \mathrm{Stab}_G\{x,y,z\} \\ |g| \text{ prime}}} 1 \\
					& = \sum_{ \substack{ g \in G \\ |g| \text{ prime}}} \frac{1}{n^3} \#\{ x,y,z \leq n : \{x,y,z\}^g = \{x,y,z\} \}.
			\end{align*}
		We therefore consider for each $g \in G$ of prime order, the number of $x,y,z \leq n$ for which $\{x,y,z\}^g = \{x,y,z\}$.  Since $g$ has prime order, the action of $g$ on $\{1,\dots,n\}$ decomposes as a union of fixed points and of cycles of length $|g|$.  It follows from this that if $\{x,y,z\}^g = \{x,y,z\}$ and $|g| \geq 5$, then in fact $x$, $y$, and $z$ must all be fixed by $g$.  Hence we find in this case that
			\[
				\#\{ x,y,z \leq n : \{x,y,z\}^g = \{x,y,z\} \}
					= \mathrm{Fix}(g)^3,
			\]
		where $\mathrm{Fix}(g)$ denotes the number of fixed points of $g$.  If $|g|=3$ and $x$, $y$, and $z$ are not all fixed by $g$, then they must comprise a $3$-cycle.  The number of $3$-cycles is $\frac{1}{3}(n-\mathrm{Fix}(g))$, and for any fixed $3$-cycle, there are $6$ possible assignments of $x,y,z$.  We conclude that if $|g| = 3$, then
			\[
				\#\{ x,y,z \leq n : \{x,y,z\}^g = \{x,y,z\} \}
					= \mathrm{Fix}(g)^3 + 2 (n- \mathrm{Fix}(g)).
			\]
		Finally, if $|g|=2$ and not all of $x$, $y$, and $z$ are fixed by $g$, then two of them must comprise a $2$-cycle, and the third must either be fixed or equal to one of the other two.  If $|g| = 2$, then there are $\frac{1}{2} (n - \mathrm{Fix}(g))$ $2$-cycles, and we find that
			\[
				\#\{ x,y,z \leq n : \{x,y,z\}^g = \{x,y,z\} \}
					= \mathrm{Fix}(g)^3 + 3 \cdot \mathrm{Fix}(g) \cdot (n- \mathrm{Fix}(g)) + 3 \cdot (n - \mathrm{Fix}(g)).
			\]
		Putting this all together, we therefore find
			\begin{align} \label{eqn:triple-probability}
				&\mathbf{Prob}\left[ \mathrm{Stab}_G\{ x, y, z\} \ne 1\right] \\
					& \quad\quad\quad\quad\leq \sum_{\substack{g \in G \\ |g| \text{ prime}}} \frac{\mathrm{Fix}(g)^3}{n^3} + \sum_{\substack{g \in G \\ |g| = 3}} \frac{2(n-\mathrm{Fix}(g))}{n^3} + \sum_{ \substack{ g \in G \\ |g|=2}} \frac{3 \cdot (\mathrm{Fix}(g)+1) \cdot (n - \mathrm{Fix}(g))}{n^3},  \notag
			\end{align}
		which is an expression that holds for any permutation group $G$ of any degree $n$.  We now specialize to the specific case of interest to us, where $G$ is the Thompson group acting in degree $n = 143\,127\,000$.  Using the ATLAS of Finite Groups \cite{ATLAS} and its implementation in GAP \cite{GAP}, we find there are $10$ conjugacy classes of elements of prime order, with labels 2A, 3A, 3B, 3C, 5A, 7A, 13A, 19A, 31A, and 31B.  We also find that $\mathrm{Fix}(\mathrm{2A}) = 10\,200$, $\mathrm{Fix}(\mathrm{3A}) = 3\,510$, $\mathrm{Fix}(\mathrm{3B}) = 243$, $\mathrm{Fix}(\mathrm{3C}) = 540$, $\mathrm{Fix}(\mathrm{7A}) = 9$, $\mathrm{Fix}(\mathrm{13A}) = 3$, and $\mathrm{Fix}(\mathrm{5A}) = \mathrm{Fix}(\mathrm{19A}) = \mathrm{Fix}(\mathrm{31A}) = \mathrm{Fix}(\mathrm{31B}) =0$, where we have written $\mathrm{Fix}(\mathrm{2A})$, for example, for the value of $\mathrm{Fix}(g)$ for any $g$ in the conjugacy class 2A.  The ATLAS also provides the size of each conjugacy class, and using this, we compute using \eqref{eqn:triple-probability} that
			\[
				\mathbf{Prob}\left[ \mathrm{Stab}_G\{ x, y, z\} \ne 1\right]
					\leq \frac{419448082}{207981421875} \approx 0.002.
			\]
		
		We now treat $\mathbf{Prob}[\mathrm{Stab}_G\{x,y,z,w\} \ne 1]$ similarly.  As above, we find
			\begin{align*}
				\mathbf{Prob}\left[ \mathrm{Stab}_G\{ x, y, z, w\} \ne 1\right]
					&\leq \sum_{ \substack{ g \in G \\ |g| \text{ prime}}} \frac{1}{n^4} \#\{ x,y,z,w \leq n : \{x,y,z,w\}^g = \{x,y,z,w\} \}.
			\end{align*}
		If $|g| \geq 5$, then 
			\[
				\#\{ x,y,z,w \leq n : \{x,y,z,w\}^g = \{x,y,z,w\} \}
					= \mathrm{Fix}(g)^4.
			\]
		If $g=3$ and $\{x,y,z,w\}^g = \{x,y,z,w\}$ and not all points are fixed, then either $\{x,y,z,w\}$ consists of a $3$-cycle and a fixed point or a $3$-cycle with a doubled point.  We thus find
			\[
				\#\{ x,y,z,w \leq n : \{x,y,z,w\}^g = \{x,y,z,w\} \}
					= \mathrm{Fix}(g)^4 + (8 \cdot \mathrm{Fix}(g)+12) \cdot (n-\mathrm{Fix}(g)).
			\]
		If $g=2$ and $\{x,y,z,w\}^g = \{x,y,z,w\}$ and not all points are fixed, then $\{x,y,z,w\}$ must form: two disjoint $2$-cycles;	a $2$-cycle plus two fixed points (possibly the same); a $2$-cycle with a double point plus a fixed point; a $2$-cycle with a triple point; or a $2$-cycle with two double points.
		Therefore, in this case, we find
			\begin{align*}
				\#\{ x,y,z,w \leq n &: \{x,y,z,w\}^g = \{x,y,z,w\} \} \\
					&= \mathrm{Fix}(g)^4  + 3\cdot(n-\mathrm{Fix}(g))^2 
					 + (12 \mathrm{Fix}(g)^2+12\mathrm{Fix}(g)+1)\cdot(n-\mathrm{Fix}(g)).
			\end{align*}
		All told, we conclude that
			\begin{align}\label{eqn:quadruple-probability}
				&\mathbf{Prob}\big[\mathrm{Stab}_G\{x,y,z,w\} \ne 1 \big] \\
				&\quad\quad\notag \leq \sum_{\substack{g \in G \\ |g| \text{ prime}}} \frac{\mathrm{Fix}(g)^4}{n^4} + \sum_{\substack{g \in G \\ |g|=3}} \frac{(8\cdot\mathrm{Fix}(g)+12)(n-\mathrm{Fix}(g))}{n^4} \\
				&\quad\quad\quad\notag + \sum_{\substack{g \in G \\ |g|=2}} \frac{3(n-\mathrm{Fix}(g))^2 + (12\cdot \mathrm{Fix}(g)^2+12\cdot\mathrm{Fix}(g)+1)(n-\mathrm{Fix}(g))}{n^4},
			\end{align}
		which, as with \eqref{eqn:triple-probability}, is an expression that holds for any permutation group $G$ of degree $n$.  Evaluating it for the Thompson group, we find that
			\[
				\mathbf{Prob}\big[\mathrm{Stab}_G\{x,y,z,w\} \ne 1\big]
					\leq \frac{2992265015279081}{5090286441648234375000}
					\approx 5.87 \cdot 10^{-7}.
			\]
		Therefore, using \eqref{eqn:union-bound}, we conclude that the probability that a random set $\Sigma$ of four points does not satisfy the conclusion of the lemma is at most $0.01$.  In particular, the conclusion of the lemma is satisfied for at least $99\%$ of the possible sets $\Sigma$, and thus for at least one.
	\end{proof}
	
	\begin{lemma}\label{lem:thompson-bound}
		Let $G$ be the Thompson group $\mathrm{Th}$.  Then Theorem~\ref{thm:explicit-almost-simple} holds for $G$, with the values $n(G)$, $a(G)$, $w(G)$, and $\gamma(G)$ recorded in Table~\ref{tab:sporadic-bounds}.  In particular, we may take $a(G) = 2.139$.
	\end{lemma}
	\begin{proof}
		Using \cite[Theorem 4.7]{LO} together with Lemma~\ref{lem:thompson-stabilizer}, we find there is a set of invariants $\{f_i\}_{i \leq n}$ with degrees $\deg f_i = i $ for $i \leq 4$, $\deg f_i = 5$ for $5 \leq i \leq n-12$, and $\deg f_i \leq 9$ for $n-11 \leq i \leq n$.  The result then follows from Lemma~\ref{lem:invariant-theory-bound}.
	\end{proof}
	
	\subsubsection{Proof of Theorem~\ref{thm:explicit-almost-simple}}
	As every possible almost simple group has been considered in Lemmas~\ref{lem:alternating-bound}--\ref{lem:thompson-bound}, the proofs of these lemmas comprise the proof of Theorem~\ref{thm:explicit-almost-simple}.  The claim that $a(G)=4$ is admissible follows by an examination of these bounds, with the largest value arising from $G=\mathrm{J}_3$.  Evaluating our bound on $\mathrm{J}_3$ exactly, we see that in fact any $a(G) \geq \frac{6935}{18 \sqrt{9690}} = 3.913\dots$ will be admissible.
	
\subsection{Proof of Theorem~\ref{thm:explicit-almost-almost-simple}}

	We now return to Theorem~\ref{thm:explicit-almost-almost-simple}.  Our first task is to convert the bounds provided by Theorem~\ref{thm:explicit-almost-simple} into the form required by Theorem~\ref{thm:explicit-almost-almost-simple}.  This conversion is principally carried out in the following two straightforward lemmas.  The first makes explicit the fact that each almost simple group $G$ has a permutation representation of degree at most $\sqrt{|G|}$.
	
	\begin{lemma} \label{lem:almost-simple-degree}
		Let $G$ be an almost simple group, and let $n(G)$ be as in Theorem~\ref{thm:explicit-almost-simple}.  Let $\alpha = \frac{171}{2\sqrt{9690}} = 0.868\dots$.  Then $n(G) \leq \alpha \sqrt{|G|}$.  This value of $\alpha$ is sharp if $G = \mathrm{J}_3$.
	\end{lemma}
	\begin{proof}
		We verify the claim directly in Magma for the sporadic groups, and for exceptional and classical groups with $q < 100$ and small rank ($\mathrm{PSL}_m$, $\mathrm{PSp}_{2m}$, $\mathrm{PSU}_m$, $\mathrm{P\Omega}_{2m+1}$, and $\mathrm{P\Omega}_{2m}^\pm$ with $m \leq 10$).  For groups of Lie type with $q>100$ or those classical groups with $m \geq 11$, it is straightforward to verify the claim.
	\end{proof}
	
	We will also make use of the following closely related result.
	
	\begin{lemma} \label{lem:almost-simple-weight-degree}
		Let $G$ be an almost simple group, and let $n(G)$ and $w(G)$ be as in Theorem~\ref{thm:explicit-almost-simple}.  Let $\beta = \frac{49 \sqrt{42}}{42} = 7.560\dots$.  Then $n(G)w(G) \leq \beta \sqrt{|G|}$.  This value of $\beta$ is sharp if $G = \mathrm{PSU}_3(\mathbb{F}_3)$.
	\end{lemma}
	\begin{proof}
		This is directly analogous to the proof of Lemma~\ref{lem:almost-simple-degree}.
	\end{proof}

	\begin{proof} [Proof of Theorem~\ref{thm:explicit-almost-almost-simple}]
		Let $N = T^r$ be the unique minimal normal subgroup of $G$, where $T$ is a finite simple group and $r \geq 1$ is an integer.  First suppose that $G$ is almost simple, i.e. that $r=1$.  Appealing to Theorem~\ref{thm:explicit-almost-simple}, we find that
			\[
				\#\mathcal{F}_k(X;G)
					\leq (2\pi)^{dn/2} (\gamma d+1)!^n |G|^{dn} (2dn^3)^{dnw} X^{\frac{a}{\sqrt{|G|}}},
			\]
		where $n=n(G)$, $a=a(G)$, $w=w(G)$, and $\gamma=\gamma(G)$ are as in Theorem~\ref{thm:explicit-almost-simple}.  We begin by noting that since $|G| \geq 60$,
			\[
				|G|^{dn}
					\leq e^{\alpha d |G|^{1/2} \log |G|}
					< e^{d |G|},
			\]
		where $\alpha$ is as in Lemma~\ref{lem:almost-simple-degree}.  Next, since $\gamma \leq 3$ in all cases, we observe that $(\gamma d + 1)! \leq (6d)^{6d}$.  From this and the fact that $\alpha < 1$, we find that
			\[
				(2\pi)^{dn/2} (\gamma d+1)!^n (2dn^3)^{dnw}
					\leq 2^{((\log 2\pi + 6 \log 6)\cdot\alpha + \beta)\cdot d|G|^{1/2}} d^{(6\alpha+\beta)\cdot d |G|^{1/2}} (|G|^2)^{\frac{3}{4}\beta \cdot d|G|^{1/2}},
			\]
		where $\beta$ is as in Lemma~\ref{lem:almost-simple-weight-degree}.  The statement of the theorem therefore holds for $G$, for any $c_1 \geq (\log 2\pi + 6 \log 6)\cdot \alpha + \beta \approx 18.495$.
		
		Now, suppose $r \geq 2$.  View $G$ as a subgroup of $\mathrm{Aut}(T) \wr S_r$, let $\pi\colon \mathrm{Aut}(T) \wr S_r \to S_r$ be the natural quotient map, let $G_r = \pi(G)$, and let $H_1 = G \cap \pi^{-1}(\mathrm{Stab}_{S_r} 1)$.  Observe that $[G:H_1] = r$.  Let $H_2 = G \cap ((1 \times \mathrm{Aut}(T)^{r-1}) \rtimes \pi(H_1))$.  Then $H_2$ is a normal subgroup of $H_1$, and $H_1 / H_2 \simeq G_0$ for some almost simple group $G_0$ with socle $T$.  Moreover, since $H_2$ does not contain $N=T^r$, the core of $H_2$ in $G$ is trivial.
		
		If $K \in \mathcal{F}_k(X;G)$, it follows that $[K^{H_1} : k ] = r$ and that $K^{H_2}/K^{H_1}$ is a Galois $G_0$-extension.  Moreover, since $H_2$ is core-free, the normal closure of the extension $K^{H_2}$ over $k$ will be $K$.  It thus follows that
			\[
				\#\mathcal{F}_k(X;G)
					\leq \sum_{F \in \mathcal{F}_{r,k}(X^{r/|G|};G_r)} \#\mathcal{F}_F(X^{\frac{r|G_0|}{|G|}};G_0).
			\]
		Since $G_0$ is almost simple, we may appeal to the $r=1$ case of the the theorem to conclude that
			\[
				\#\mathcal{F}_F(X^{\frac{r|G_0|}{|G|}};G_0)
					\leq e^{dr|G_0|} (2 dr |G_0|^2)^{c_1 d r |G_0|^{1/2}} X^{\frac{cr\sqrt{|G_0|}}{|G|}},
			\]
		while using \cite[Theorem 2.19]{LO}, we find
			\[
				\#\mathcal{F}_{r,k}(X^{\frac{r}{|G|}};G_r)
					\leq (2\pi)^{d(r-1)/2} (d+1)!^{r-1} r^{\frac{d(5r-2)}{4}+1} X^{\frac{r(r+2)}{4|G|}}.
			\]
		For convenience, we collect the powers of $X^{1/\sqrt{|G|}}$ and the constants separately.  We first note that $|G| \geq r|G_0|T^{r-1}$, so that
			\begin{align*}
				X^{\frac{cr\sqrt{|G_0|}}{|G|} + \frac{r(r+2)}{4|G|}}
					&=(X^{\frac{1}{\sqrt{|G|}}})^{\frac{cr \sqrt{|G_0|}}{\sqrt{|G|}} + \frac{r(r+2)}{4 \sqrt{|G|}}} \\
					&\leq (X^{\frac{1}{\sqrt{|G|}}})^{\frac{c \sqrt{r}}{|T|^{(r-1)/2}} + \frac{(r+2)\sqrt{r}}{4|T|^{r/2}} } \\
					&\leq (X^{\frac{1}{\sqrt{|G|}}})^{\frac{c \sqrt{2}}{\sqrt{60}} + \frac{\sqrt{2}}{60}} \\
					& \leq X^{\frac{c}{\sqrt{|G|}}}
			\end{align*}
		for any $c \geq \frac{\sqrt{2}}{60}\left(1-\frac{\sqrt{2}}{{\sqrt{60}}} \right)^{-1} = 0.0288\dots$.
		
		For the constants, we note first that $r \sqrt{|G_0|} \leq \frac{1}{\sqrt{30}} \sqrt{|G|}$, which implies that
			\[
				e^{dr|G_0|} (2 dr |G_0|^2)^{c_1 d r |G_0|^{1/2}}
					\leq e^{d|G|} (2 d |G|^2)^{\frac{c_1}{\sqrt{30}} d |G|^{1/2}},
			\]
		and on also using $r \leq \frac{\sqrt{2}}{60} \sqrt{|G|}$, that
			\[
				 (2\pi)^{d(r-1)/2} (d+1)!^{r-1} r^{\frac{d(5r-2)}{4}+1}
					\leq 2^{d|G|^{1/2} \cdot 
						\left(\frac{\log 2\pi}{2}\cdot\frac{\sqrt{2}}{60} + \frac{\sqrt{2}}{30} - \frac{\sqrt{2}}{40\log 2}\cdot \log\left(\frac{60}{\sqrt{2}}\right) \right) } d^{d|G|^{1/2} \frac{\sqrt{2}}{30}} (|G|^2)^{d|G|^{1/2} \frac{\sqrt{2}}{160}}.
			\]
		Pulling this together, we find that
			\[
				\#\mathcal{F}_k(X;G)
					\leq e^{d |G|} (2d |G|^2)^{c_1 d |G|^{1/2}} X^{\frac{c}{\sqrt{|G|}}}
			\]
		provided that $c_1 \geq \frac{\sqrt{2}}{30}\left(1-\frac{1}{\sqrt{30}}\right)^{-1} = 0.057\dots$.  The limiting bounds therefore arise from the almost simple case, and the theorem follows.
	\end{proof}
	
	\subsection{Proofs of Theorems~\ref{thm:explicit-galois-bound-intro}--\ref{thm:optimal-galois-bound-intro}}
	
	We observe that Theorem~\ref{thm:explicit-almost-almost-simple} completes the proofs of Theorem~\ref{thm:explicit-galois-bound-intro} and Theorem~\ref{thm:uniform-galois-bound-intro}, when combined with Lemma~\ref{lem:reduction} and the previous work on groups all of whose minimal normal subgroups are abelian.  Moreover, an analysis of the proof of Theorem~\ref{thm:explicit-almost-almost-simple} (particularly that the worst-case is provided by almost simple groups) and the bounds provided by Lemmas~\ref{lem:alternating-bound}--\ref{lem:thompson-bound} (or Tables~\ref{tab:classical-bounds}--\ref{tab:sporadic-bounds}) implies that to prove Theorem~\ref{thm:optimal-galois-bound-intro}, we must only prove that it holds for the group $G=\mathrm{J}_3$.  We therefore have the following lemma.
	
	\begin{lemma} \label{lem:J3-optimal}
		Let $G = \mathrm{J}_3$ and let $c_0 = \frac{863441}{2880\sqrt{9690}} \approx 3.045$.  Then for any number field $k$, any $X \geq 1$, and any $\epsilon > 0$, we have
			\[
				\#\mathcal{F}_k(X;G)
					\ll_{k,\epsilon} X^{\frac{c_0}{\sqrt{|G|}} + \epsilon}.
			\]
	\end{lemma}
	\begin{proof}
		As in the proof of Lemma~\ref{lem:J1-J3}, we find that in its degree $n:=6156$ primitive permutation representation, there is a set of independent $G$-invariants $\{f_i\}$ with degrees $\deg f_i = i$ for $i \leq 4$, $\deg f_i = 5$ for $5 \leq i \leq 6144$, and $\deg f_i \leq 9$ for $6145 \leq i \leq 6156$.  We also compute that the index of $G$ (in the sense of Malle) is $3040$.  Using \cite[Theorem 3.16]{LO}, we then see that
			\[
				\#\mathcal{F}_{6156,k}(X;\mathrm{J}_3)
					\ll_{k,\epsilon} X^{\frac{863441}{246240}+\epsilon},
			\]
		where the bound is on the number of degree $6156$ $\mathrm{J}_3$-extensions.  Since $|\mathrm{J}_3| = 50\,232\,960 = 2^7 \cdot 3^5 \cdot 5 \cdot 17 \cdot 19$, using Lemma~\ref{lem:non-galois-passage}, we therefore find
			\[
				\#\mathcal{F}_k(X;\mathrm{J}_3)
					\ll_{k,\epsilon} X^{\frac{863441}{2009318400} + \epsilon}.
			\]
		This is exactly the statement of the lemma.
	\end{proof}
	
	This completes the proof of Theorem~\ref{thm:optimal-galois-bound-intro}.

\subsection{An asymptotic improvement to the shape of Theorem~\ref{thm:explicit-almost-almost-simple}}
	Finally, we note that while the sporadic groups (in particular $\mathrm{J}_3$) are the bottleneck in computing the explicit power of $X$ provided by Theorem~\ref{thm:explicit-almost-almost-simple}, they \emph{cannot} be the bottleneck in the asymptotic shape of the exponent, since there are only finitely many such almost simple groups (which the proof shows dominate the bounds).  We make this clear in the following theorem, which we prove in a soft form to make the ideas clearer and to avoid the need for extensive case work.
	
	\begin{theorem} \label{thm:optimal-almost-almost-simple}
		There is a constant $c>0$ such that the following holds. Let $G$ be a finite group with a unique minimal normal subgroup $N$, and suppose that $N$ is not abelian.  Then for any number field $k$ and any $X \geq 1$, there holds
			\[
				\#\mathcal{F}_k(X;G)
					\ll_{k,G} X^{ \frac{c}{|G|^{4/7}}}.
			\]
	\end{theorem}
	\begin{proof}
		We suppose first that $G$ is almost simple.  For each possible group of Lie type (regarding the prime power $q$ as a parameter), we compute the least exponent $\delta>0$ such that
			\begin{equation} \label{eqn:weighted-degree-lie-type}
				n(G) \cdot w(G) \ll |G|^{\delta}
			\end{equation}
		as $q \to \infty$, where $n(G)$ and $w(G)$ are as in Theorem~\ref{thm:explicit-almost-simple}.  It is an exercise to see that the largest such $\delta$ arises from groups of the form $\mathbin{^2G_2}(q)$.  In particular, \eqref{eqn:weighted-degree-lie-type} holds with $\delta = \frac{3}{7}$ for every almost simple group of Lie type, and, since every $\delta>0$ is admissible for the alternating groups $A_n$ as $n\to\infty$, hence also for every almost simple group.  Using Lemma~\ref{lem:non-galois-passage}, Lemma~\ref{lem:invariant-theory-bound}, and Lemma~\ref{lem:induction-bound} as above, we conclude that there is some constant $C>0$ such that
			\[
				\#\mathcal{F}_k(X;G)
					\ll_{k,G} X^{\frac{C}{|G|^{1-\delta}}}
					\ll X^{\frac{C}{|G|^{4/7}}}.
			\]
		
		We now suppose that $N = T^r$ for some $r \geq 2$ and some simple group $T$.  Letting $\delta$ be admissible in \eqref{eqn:weighted-degree-lie-type} for almost simple groups with socle $T$, we see as in the proof of Theorem~\ref{thm:explicit-almost-almost-simple} that
			\[
				\#\mathcal{F}_k(X;G)
					\ll_{k,G} X^{\frac{r(r+2)}{4|G|} + \frac{r |G_0|^\delta}{|G|}} 
					= X^{\frac{1}{|G|^{1-\delta}} \left( \frac{r(r+2)}{4|G|^\delta} + \frac{r |G_0|^\delta}{|G|^\delta} \right)}
					\ll X^{\frac{C^\prime}{|G|^{1-\delta}}}
					\ll X^{\frac{C^\prime}{|G|^{4/7}}}
			\]
		for some $C^\prime > 0$.  Hence, taking $c = \max\{ C, C^\prime\}$, the result follows.
	\end{proof}

\section{The proof of Theorems \ref{thm:galois-count-intro} and \ref{thm:normal-closure-count-intro}}
	\label{sec:proof-count}

	In this section, we prove Theorems \ref{thm:galois-count-intro} and \ref{thm:normal-closure-count-intro}.  In carrying this out, we find it convenient to first recall some useful results.

\subsection{Bounds on discriminants and the number of finite groups}
	
	To approach Theorem~\ref{thm:galois-count-intro}, we will make use of the following consequence of work of Odlyzko \cite{Odlyzko}, that we state in a slightly simplified form.  This will limit the groups $G$ contributing nontrivially to $\#\mathcal{F}_k^\mathrm{Gal}(X)$ and $\#\mathcal{F}_k^\mathrm{nc}(X)$.
	
	\begin{lemma}\label{lem:odlyzko-bound}
		There is an absolute constant $C$ such that for any number field $k$ and any extension $K/k$, we have $dn \leq \frac{1}{3} \log |\mathrm{Disc}(K)| + C$, where $d=[k:\mathbb{Q}]$ and $n = [K:k]$.
	\end{lemma}
	\begin{proof}
		By \cite[Theorem 1]{Odlyzko}, we have $|\mathrm{Disc}(K)| \geq 21^{dn}$ for $dn$ sufficiently large.  Since $\log 21 = 3.044\dots > 3$, this implies the result for $dn$ sufficiently large, say $dn \geq N_0$.  The constant $C$ is chosen so that the statement remains true for the finitely many values of $dn$ that are less than $N_0$.  For example, $C=N_0$ is admissible.
	\end{proof}
	\begin{remark}
		Odlyzko's theorem is stated as a lower bound on $|\mathrm{Disc}(K)|$ in terms of the degree $[K:\mathbb{Q}]$, but for us it is convenient to state the result in terms of the two parameters $d$ and $n$, even if this is somewhat artificial.
	\end{remark}
	
	As a consequence of Lemma \ref{lem:odlyzko-bound}, the groups $G$ contributing to $\mathcal{F}_k^\mathrm{Gal}(X)$ all satisfy $d |G| \leq \frac{1}{3} \log X + C$.  This puts a bound on how complicated the individual groups appearing may be, but to prove Theorem \ref{thm:galois-count-intro} and \ref{thm:normal-closure-count-intro}, we shall also need the following result of Holt \cite{Holt} that provides an upper bound on the number of groups of bounded order.
	
	\begin{lemma}\label{lem:number-of-groups}
		For any $N \geq 1$, the number of isomorphism classes of finite groups $G$ with $|G| \leq N$ is bounded above by $N^{\frac{(\log N)^2}{6 (\log 2)^2} + \frac{\log N}{\log 2}}$.
	\end{lemma}
	\begin{proof}
		This is \cite[Theorem 2]{Holt}, which we note relies on the classification of finite simple groups.
	\end{proof}
	
	In particular, combining Lemmas \ref{lem:odlyzko-bound} and \ref{lem:number-of-groups}, we see that the number of groups $G$ contributing nontrivially to $\mathcal{F}_k^\mathrm{Gal}(X)$ is $\leq \exp( O((\log\log X)^3)) = X^{o(1)}$.  This is a sufficiently small number that there will be essentially no difficulty in adding together the contributions from the different groups $G$.  We now make this explicit by proving Theorems \ref{thm:galois-count-intro} and \ref{thm:normal-closure-count-intro}.

\subsection{Proof of Theorems \ref{thm:galois-count-intro} and \ref{thm:normal-closure-count-intro}}

	We begin with an easy lemma that will be used to control the discrepancy between $\#\mathcal{F}_k^{\mathrm{nc}}(X)$ and $\#\mathcal{F}_k^{\mathrm{Gal}}(X)$.
	
	\begin{lemma}\label{lem:core-free}
		For any finite group $G$, let $\mathrm{CoreFree}(G) := \#\{ H \leq G : \cap_{g \in G} H^g = 1\}$ denote the number of core-free subgroups $H \leq G$.
		Then $\mathrm{CoreFree}(G) \leq \exp\left( \frac{(\log |G|)^2}{\log 2} \right)$.
	\end{lemma}
	\begin{proof}
		Trivially, we may bound $\mathrm{CoreFree}(G)$ by the number of subgroups $H \leq G$ (without regard to whether they are core-free).  Any subgroup $H \leq G$ is generated by its Sylow subgroups $H_p$.  Each Sylow subgroup $H_p$ is generated by its center $Z(H_p)$ and generators of the quotient $H_p / Z(H_p)$, from which follows that $H_p$ is generated by at most $v_p(|H_p|)$ elements.  Hence, any subgroup $H \leq G$ is generated by at most $\Omega(|G|)$ elements, where $\Omega(|G|)$ is the number of prime divisors of $|G|$, counted with multiplicity.
		
		From this, we find 
			\[
				\mathrm{CoreFree}(G) \leq |G|^{\Omega(|G|)} \leq \exp\left( \frac{(\log |G|)^2}{\log 2} \right),
			\]
		as claimed.
	\end{proof}
	
	\begin{proof}[Proof of Theorem \ref{thm:normal-closure-count-intro}]
		We first observe that for any $X \geq e^e$ (so that $\log\log X \geq 1$), we have by Lemma \ref{lem:odlyzko-bound}
			\begin{align*}
				\#\mathcal{F}_k^{\mathrm{nc}}(X)
					&= \sum_{|G| \leq 4} \#\mathcal{F}_k(X;G) + \sum_{5 \leq |G| \leq \frac{1}{3d} \log X + \frac{C}{d}} \#\mathcal{F}_k(X;G) \cdot \mathrm{CoreFree}(G) \\
					&= \sum_{|G| \leq 4} \#\mathcal{F}_k(X;G) + \sum_{5184 \leq |G| \leq \frac{1}{3d} \log X + \frac{C}{d}}\#\mathcal{F}_k(X;G) \cdot \mathrm{CoreFree}(G) + O_{k,\epsilon}(X^{\frac{3}{8}+\epsilon}),
			\end{align*}
		where we have invoked \cite[Proposition 1.3]{EllenbergVenkatesh} to obtain the term $O_{k,\epsilon}(X^{\frac{3}{8}+\epsilon})$.  
		
		Note that $\frac{6}{\sqrt{n}} \leq \frac{1}{12}$ for $n \geq 5184$.  For any $G$ with $|G| \leq \frac{1}{3d} \log X + \frac{C}{d}$, Lemma \ref{lem:core-free} shows that $\mathrm{CoreFree}(G) \leq \exp( O( (\log\log X)^2))$, while Lemma \ref{lem:number-of-groups} shows that the number of groups $G$ with order $5184 \leq |G| \leq \frac{1}{3d} \log X + \frac{C}{d}$ is $\exp( O( (\log\log X)^3))$, where all implied constants are absolute.  Additionally, by Theorem \ref{thm:explicit-galois-bound-intro}, for groups $G$ such that $5184 \leq |G| \leq \frac{1}{3d} \log X + \frac{C}{d}$, we find
			\[
				\#\mathcal{F}_k(X;G) \ll X^{\frac{1}{3}} e^{O((\log X)^{1/2} (\log\log X))} X^{\frac{6}{\sqrt{n}}} \ll_\epsilon X^{\frac{5}{12}+\epsilon}.
			\]
		All told, we find
			\[
				\sum_{5184 \leq |G| \leq \frac{1}{3d} \log X + \frac{C}{d}}\#\mathcal{F}_k(X;G) \cdot \mathrm{CoreFree}(G)
					\ll_{\epsilon} X^{\frac{5}{12}+\epsilon}
			\]
		for any $\epsilon > 0$, and therefore also
			\[
				\#\mathcal{F}_k^{\mathrm{nc}}(X)
					= \sum_{|G| \leq 4} \#\mathcal{F}_k(X;G) + O_{k,\epsilon}(X^{\frac{5}{12} + \epsilon}).
			\]
		If $k \ne \mathbb{Q}$, then for $G = C_3$, $C_4$, and $C_2 \times C_2$, the bound $\#\mathcal{F}_k(X;G) \ll_{k,\epsilon} X^{\frac{1}{2}+\epsilon}$ follows from work of Wright \cite{Wright}.  We thus find
			\[
				\#\mathcal{F}_k^\mathrm{nc}(X)
					= \#\mathcal{F}_k(X;C_2) + O_{k,\epsilon}(X^{\frac{1}{2}+\epsilon}).
			\]
		The result in this case now follows from \cite[Theorem 2]{McgownTucker}.
		
		If $k=\mathbb{Q}$, then we have the stronger asymptotic
			\[
				\#\mathcal{F}_\mathbb{Q}(X;C_2)
					= \frac{6}{\pi^2} X + O(X^{1/2} \exp(-c \cdot (\log X)^{3/5} (\log\log X)^{-1/5}))
			\]
		for some $c>0$, as follows from known zero-free regions for $\zeta(s)$ and $L(s,\chi_4)$ and standard techniques; in fact, using recent work of Khale \cite{Khale}, it is possible to provide an explcit value of $c$.  Moreover, it follows from \cite[Theorem 1.7]{FreiLoughranNewton} that there is some $\delta>0$ such that
			\[
				\#\mathcal{F}_\mathbb{Q}(X;C_3) + \#\mathcal{F}_\mathbb{Q}(X;C_4) + \#\mathcal{F}_\mathbb{Q}(X;C_2 \times C_2)
					= P_2(\log X) \cdot X^{\frac{1}{2}} + O(X^{\frac{1}{2} - \delta}),
			\]
		where $P_2(\log X)$ is an explicitly computable polynomial of degree $2$ in $\log X$, completing the proof of the theorem when $k=\mathbb{Q}$.
	\end{proof}
	
	\begin{proof}[Proof of Theorem~\ref{thm:galois-count-intro}]
		This follows mutatis mutandis from the proof above by removing the weight $\mathrm{CoreFree}(G)$.
	\end{proof}
	
	\subsection{Variants} \label{sec:variants} For most variations of Theorem~\ref{thm:galois-count-intro} incorporating a weighting that decays (e.g., weighing fields inversely to the size of their automorphism group, or looking at fields only up to isomorphism), or that is $O_\epsilon(X^\epsilon)$ uniformly for $|G| \leq \frac{1}{3d} \log X + C$, the proof of Theorem~\ref{thm:galois-count-intro} shows that to obtain an asymptotic formula, it suffices to understand the weighting only for $G=C_2$ and to modify the asymptotic accordingly.    

\section{Bounds without the classification of finite simple groups}
	\label{sec:no-cfsg}

	In this final section, we provide a proof of the bound in Theorem~\ref{thm:galois-bound-no-cfsg}, which does not rely on the classification of finite simple groups.  Analogous to our treatment of groups with a nonabelian socle, we exploit the methods of \cite{LO} relying on algebraically independent invariants.  The key result is the following.
	
	\begin{theorem} \label{thm:regular-invariants}
		Let $G$ be a finite group, say of order $n$, and regard $G$ as a permutation group of degree $n$ via its regular representation.  Then there is an algebraically independent set $\{f_1,\dots,f_n\} \subseteq \mathbb{Z}[x_1,\dots,x_n]^G$ of $G$-invariants satisfying:
			\begin{itemize}
				\item $\deg f_1 = 1$;
				\item $\deg f_i = 2$ for $2 \leq i \leq \frac{n + n_2+1}{2}$, where $n_2$ is the number of elements of $G$ with order $2$; and
				\item $\deg f_i = 3$ for $\frac{n+n_2+1}{2}+1 \leq i \leq n$.
			\end{itemize}
		This set is minimal, in the sense that any other set of algebraically independent invariants $\{f_1^\prime,\dots,f_n^\prime\}$ with $\deg f_i^\prime \leq \deg f_{i+1}^\prime$ for $1 \leq i \leq n-1$ must satisfy $\deg f_i^\prime \geq \deg f_i$ for every $i \leq n$.
	\end{theorem}
	\begin{proof}
		We begin by describing the invariants $f_i$ explicitly.  To this end, for a polynomial $f \in \mathbb{Z}[x_1,\dots,x_n]$, let $f^G := \sum_{g \in G} f^g$, which is necessarily $G$-invariant.  We begin by setting $f_1 := x_1^G = x_1 + \dots + x_n$.  We next consider invariants of the form $(x_1x_i)^G$ for $i \ne 1$, and claim that there are $\frac{n+n_2-1}{2}$ distinct invariants of this form, which we take to be the invariants $f_2,\dots,f_{(n+n_2+1)/2}$.  In particular, if $(x_1x_i)^G = (x_1x_j)^G$ for some $i \ne j$, then there is some $g \in G$ such that $(x_1x_i)^g = x_1x_j$.  Since $G$ acts regularly, we cannot have $x_1^g = x_1$, so we must have that $x_i^g = x_1$ and $x_1^g = x_j$.  In particular, $g$ is the unique element of $G$ sending $i$ to $1$, and $j$ is the image of $1$ under the action of $g$.  Write $g_i$ for this element.  Since $i \ne j$, we conclude that $(x_1x_i)^G = (x_1x_j)^G$ for some $j \ne i$ precisely when $g_i$ has order greater than $2$.  There are therefore $\frac{n-1-n_2}{2} + n_2 = \frac{n+n_2-1}{2}$ such invariants.  Finally, for those $i$ for which the element $g_i$ as above has order greater than $2$, we also add the invariant $(x_1^2 x_i)^G$, and we take these invariants to comprise the remaining $f_i$.
		
		We next show that these invariants are algebraically independent, for which it suffices to show that the determinant of the associated Jacobian matrix is a non-zero polynomial.  We therefore consider the partial derivatives of the invariants above, and we claim that in the expression for the determinant as a signed sum over permutations, there is a unique permutation giving rise to a monomial with a maximal power of $x_1$.  To see this, suppose first that $i$ is such that the element $g_i$ has order $2$, and let $\phi = (x_1x_i)^G$.  In this case, the partial derivative $\frac{\partial \phi}{\partial x_i}$ is the unique partial derivative containing the monomial $x_1$.  For $i \ne 1$ such that the element $g_i$ has order greater than $2$, let $\phi_1 = (x_i x_1)^G$ and $\phi_2 = (x_1^2 x_i)^G$.  Unlike the case where $g_i$ has order $2$, there are two partial derivatives of $\phi_1$ containing the monomial $x_1$, namely $\frac{\partial \phi_1}{\partial x_i}$ and $\frac{\partial \phi_1}{\partial x_j}$, where $j$ is the image of $1$ under $g_i$.  However, the partial derivative $\frac{\partial \phi_2}{\partial x_i}$ is the only partial derivative of $\phi_2$ containing the monomial $x_1^2$.  It follows that, from the rows of the Jacobian corresponding to $\phi_1$ and $\phi_2$, one must choose the partial derivatives $\frac{\partial \phi_1}{\partial x_j}$ and $\frac{\partial \phi_2}{\partial x_i}$ to obtain a maximal power of $x_1$.  Finally, since all other partial derivatives have been exhausted, it follows that we must choose the partial derivative $\frac{\partial f_1}{\partial x_1}$ of the invariant $f_1 = (x_1)^G$, yielding the claim about the Jacobian and hence the theorem.
		
		Finally, to see the claim about minimality, we observe that every monomial of degree $2$ will be present in some $(x_1x_i)^G$ above except those of the form $x_i^2$ for some $i$.  Letting $f_0 = (x_1^2)^G$, it follows that the set $\{f_0,f_1,\dots,f_{(n+n_2+1)/2}\}$ forms a basis for the vector space of $G$-invariants with degree at most $2$.  But this set is not algebraically independent, as for example $f_1^2 = f_0 + 2 \sum_{i=2}^{(n+n_2+1)/2} f_i$.  It follows that any algebraically independent set of $G$-invariants cannot contain more than $1$ invariant of degree $1$ nor more than $\frac{n+n_2-1}{2}$ invariants of degree $2$, which gives the claim.
	\end{proof}
	
	As a consequence, we find the following.
	
	\begin{corollary}\label{cor:regular-invariant-bound}
		Let $G$ be a finite group, let $n = |G|$, and let $n_2$ be the number of elements of $G$ with order $2$.  Let $p$ be the least prime dividing $n$.  Then for any number field $k$, any $X \geq 1$, and any $\epsilon>0$, we have
			\[
				\#\mathcal{F}_k(X;G)
					\ll_{k,G,\epsilon} X^{1 - \frac{n_2}{2n} - \frac{3}{2n} + \frac{p}{(p-1)n} + \epsilon}.
			\]
	\end{corollary}
	\begin{proof}
		Let $\mathcal{I}$ be the set of invariants provided by Theorem~\ref{thm:regular-invariants}, and observe that 
			\[
				\deg \mathcal{I} := \sum_{i=1}^n \deg f_i = 1 + 2 \cdot \frac{n+n_2-1}{2} + 3 \cdot (n - \frac{n + n_2 + 1}{2}) = \frac{5}{2}n - \frac{1}{2} n_2 - \frac{3}{2}.
			\]
		Since $G$ is in its regular representation, we have $\mathrm{ind}(G) = \frac{p-1}{p} n$.  It therefore follows from \cite[Theorem 3.16]{LO} that
			\[
				\#\mathcal{F}_k(X;G)
					\ll_{k,G,\epsilon} X^{1 - \frac{n_2}{2 n} - \frac{3}{2n} + \frac{p}{(p-1)n} + \epsilon},
			\]
		as claimed.
	\end{proof}
	
	With this, we are now ready to prove Theorem~\ref{thm:galois-bound-no-cfsg}.
	
	\begin{proof}
		Let $G$ be a finite group, and let $n$, $n_2$, and $p$ be as in Corollary~\ref{cor:regular-invariant-bound}.  We first observe that if $p \geq 5$, then Corollary~\ref{cor:regular-invariant-bound} yields
			\[
				\#\mathcal{F}_k(X;G)
					\ll_{k,G,\epsilon} X^{1 - \frac{3}{2n} + \frac{p}{(p-1)n} + \epsilon}
					\ll_{k,G,\epsilon} X^{1 - \frac{1}{4n} + \epsilon},
			\]
		which is sufficient.  If $p \leq 3$, then (noting that $n_2$ must be at least $1$ if $p=2$), we obtain
			\begin{equation} \label{eqn:initial-no-cfsg}
				\#\mathcal{F}_k(X;G)
					\ll_{k,G,\epsilon} X^{1 + \epsilon},
			\end{equation}
		which is not quite sufficient for the claim but will still be of use to us.  Suppose first that $p=2$.  If $n_2 \geq 2$, then in fact Corollary~\ref{cor:regular-invariant-bound} yields
			\[
				\#\mathcal{F}_k(X;G)
					\ll_{k,G,\epsilon} X^{1 - \frac{1}{2n} + \epsilon},
			\]
		which is sufficient.  Thus, if $p=2$, we need only consider the case that $n_2 = 1$.  However, if $n_2 = 1$, then the unique element of order $2$ must be central (since it is equal to its conjugates).  Letting $A \leq Z(G)$ be the corresponding subgroup of order $2$, we find by Lemma~\ref{lem:central-extensions} and \eqref{eqn:initial-no-cfsg} that
			\[
				\#\mathcal{F}_k(X;G)
					\ll_{k,G,\epsilon} X^{\frac{2}{n} + \epsilon} \cdot \#\mathcal{F}_k(X^{1/2};G/A)
					\ll_{k,G,\epsilon} X^{\frac{1}{2} + \frac{2}{n} + \epsilon}.
			\]
		This is sufficient unless $n=4$, in which case either $G \simeq C_2 \times C_2$ or $G \simeq C_4$.  However, in both of these cases, it is known that $\#\mathcal{F}_k(X;G) \ll_{k,G,\epsilon} X^{\frac{1}{2}+\epsilon}$, which completes the proof if $p=2$.
		
		Finally, if $p=3$, then it follows by the Feit--Thompson theorem that $G$ must be solvable.  Hence, if $N \unlhd G$ is a minimal normal subgroup, then $N \simeq \mathbb{F}_\ell^r$ for some odd prime $\ell$.  By considering cyclic degree $\ell$ extensions of the fields fixed by $N$, we see by Corollary~\ref{cor:inexplicit-abelian-bound} and \eqref{eqn:initial-no-cfsg} that
			\[
				\#\mathcal{F}_k(X;G)
					\ll_{k,G,\epsilon} X^{\frac{\ell}{(\ell-1)|N|} + \frac{1}{|N|} + \epsilon}
					\ll_{k,G,\epsilon} X^{\frac{1}{\ell - 1} + \frac{1}{\ell} + \epsilon}
					\ll_{k,G,\epsilon} X^{\frac{5}{6}+\epsilon},
			\]
		which is more than sufficient.
	\end{proof}
		
\bibliographystyle{alpha}
\bibliography{references}

\end{document}